\documentclass[11pt,reqno,oneside,a4paper]{amsart}
\usepackage[utf8]{inputenc}
\usepackage{amsfonts, amsmath, amssymb}
\usepackage{amsthm}
\usepackage{mathtools}
\usepackage{fullpage}
\usepackage{xcolor}
\usepackage{tikz-cd}
\usepackage{scalerel}
\usepackage{hyperref}
\usetikzlibrary{calc}
\usetikzlibrary{shapes.geometric}
\usetikzlibrary{decorations.pathmorphing}
\usepackage[all]{xy}
\usepackage[shortlabels]{enumitem}
\usepackage{bm}
\usepackage[normalem]{ulem} 
\usepackage{mathrsfs}
\usepackage{comment}
\usepackage{accents} 
\usepackage[scr=boondox]{mathalpha}
\usepackage{array}
\newcolumntype{C}[1]{>{\centering\arraybackslash}p{#1}}

\newcommand{\sfw}{\mathsf{w}}

\let\oldLeftrightarrow\Leftrightarrow
\renewcommand{\Leftrightarrow}{\mathrel{\;\oldLeftrightarrow\;}}

\newtheorem{theorem}[equation]{Theorem}
\newtheorem{lemma}[equation]{Lemma}

\newtheorem{proposition}[equation]{Proposition}

\theoremstyle{definition}
\newtheorem{defn}[equation]{Definition}

\newtheorem{example}[equation]{Example}

\newtheorem*{theorem*}{Theorem}

{
\theoremstyle{definition}
\newtheorem*{example*}{Example}
}

\newcommand{\DMO}{\DeclareMathOperator}

\newcommand{\beq}{\begin{equation}}
\newcommand{\eeq}{\end{equation}}

\newcommand\restr[2]{{
  \left.\kern-\nulldelimiterspace 
  #1 
  \right|_{#2} 
  }}
  
\renewcommand{\leq}{\leqslant}
\renewcommand{\geq}{\geqslant}

\DeclareMathOperator{\id}{{id}}

\DeclareMathOperator{\Ext}{{Ext}}

\DeclareMathOperator{\Sym}{{Sym}}

\DeclareMathOperator{\spann}{span}

\DeclareMathOperator{\hilb}{hilb}

\DeclareMathOperator{\sgn}{sgn}
\DeclareMathOperator{\inv}{inv}
\DeclareMathOperator{\NF}{NF}

\newcommand{\ZZ}{{\mathbb Z}}

\newcommand{\NN}{{\mathbb N}}

\DeclareMathOperator{\GKdim}{GKdim}
\DMO{\wt}{wt}
\DMO{\gldim}{gldim}

\newcommand{\tikzAngleOfLine}{\tikz@AngleOfLine}
\def\tikz@AngleOfLine(#1)(#2)#3{%
\pgfmathanglebetweenpoints{%
\pgfpointanchor{#1}{center}}{%
\pgfpointanchor{#2}{center}}
\pgfmathsetmacro{#3}{\pgfmathresult}%
}

\newcommand{\exterior}{\mathop{\raisebox{1pt}{$\bigwedge$}}\nolimits}

\DMO{\supp}{supp}
\renewcommand{\geq}{\geqslant}
\renewcommand{\leq}{\leqslant}
\DMO{\Rep}{Rep}

\DeclareMathOperator{\GL}{GL}

\makeatletter
\newcommand{\vast}{\bBigg@{2.5}}
\newcommand{\Vast}{\bBigg@{5}}
\makeatother

\makeatletter
\newcommand{\leqnomode}{\tagsleft@true\let\veqno\@@leqno}
\makeatother

\allowdisplaybreaks 

\title{Dual Reflection Groups from Solutions to the Yang--Baxter Equation}
\author{Simon Crawford}
\address{The University of Manchester, Alan Turing Building, Oxford Road, Manchester, M13 9PL, United Kingdom}
\email{simon.crawford@manchester.ac.uk}


\begin{document}

\begin{abstract}
Suppose that $r: X^2 \to X^2$ is a nondegenerate, involutive solution to the set-theoretic Yang--Baxter equation. To this data, we can associate an Artin--Schelter regular algebra $A$ and an infinite family of groups $G_k$, where $k \geqslant 1$. For each $k$, we show that $G_k$ is a dual reflection group for $A$, in the sense the component of $A$ graded by $1_{G_k}$ is again Artin--Schelter regular. We also show that, using the additional structure afforded by the underlying solution to the Yang--Baxter equation, one can determine the centre, Koszul complex, and associated twisted superpotential for $A$.
\end{abstract}

\maketitle

\setcounter{tocdepth}{3}
\let\oldtocsection=\tocsection
\let\oldtocsubsection=\tocsubsection
\let\oldtocsubsubsection=\tocsubsubsection
\renewcommand{\tocsection}[2]{\hspace{0em}\oldtocsection{#1}{#2}}
\renewcommand{\tocsubsection}[2]{\hspace{1.5em}\oldtocsubsection{#1}{#2}}
\renewcommand{\tocsubsubsection}[2]{\hspace{3em}\oldtocsubsubsection{#1}{#2}}



\numberwithin{equation}{section}
\section{Introduction}

One of the recent major focuses of noncommutative invariant theory has been the search for generalisations of the Chevalley--Shephard--Todd Theorem to actions of Hopf algebras on noncommutative algebras. The classical Chevalley--Shephard--Todd Theorem from commutative algebra says that, if $G$ is a finite group of graded automorphisms acting on a polynomial ring $R$, then the invariant ring $R^G$ is also a polynomial ring if and only if $G$ is generated by reflections. 

In noncommutative invariant theory, the role of a ``noncommutative polynomial ring'' is typically played by an Artin--Schelter (AS) regular algebra. However, due to the inherent rigidity of noncommutative algebras, their symmetries are generally not adequately captured by groups, and we must also allow for ``quantum symmetries'' afforded by Hopf algebras; see \cite{etingof}. Accordingly, we say that a (typically finite-dimensional and semisimple) Hopf algebra is a \emph{reflection Hopf algebra} for an AS regular algebra $A$ if the invariant ring $A^H$ is again AS regular. See \cite{three, ferraro} for examples of this phenomenon.

Of particular recent interest has been the case where $H = (\Bbbk G)^*$ is the dual of a finite-dimensional group algebra. An action of such a Hopf algebra on an AS regular algebra $A$ is equivalent to $A$ possessing a $G$-grading. In this case, the invariant ring $A^H$ is equal to the set of elements of $G$-degree $1_G$, and is called the \emph{subring of dual invariants}. When this ring is AS regular, we say that $G$ is a \emph{dual reflection group} for $A$. Examples of dual reflection groups have been studied in \cite{kirkmannakayama,chenrigidity,coactions}.

The motivation for the present work stemmed initially from a desire to generalise some of the results of \cite{goetz}, in which the authors identified two groups of order $16$ which are dual reflection groups for certain four-dimensional AS regular algebras. The main result of this paper, Theorem \ref{thm:IntroThm1}, shows that these examples fit into a much broader framework. The hidden feature of the examples in \cite{goetz} is that they arise from solutions to the Yang--Baxter equation; we show that \emph{every} such solution gives rise to an infinite family of dual reflection groups, giving a wealth of interesting examples. Moreover, the additional structure coming from the Yang--Baxter equation makes it much easier to establish AS regularity of the algebras studied in \cite{goetz}, as it allows one to apply the results of Gateva-Ivanova--Van den Bergh, \cite{gateva}, on semigroup algebras of $I$-type.

We briefly provide some context for the main result of \cite{gateva}, before stating it in a paraphrased form. If $X$ is a finite set and $r : X^2 \to X^2$ is a map, then we say that $r$ is a \emph{(set-theoretic) solution to the Yang--Baxter equation} if 
\begin{align*}
(r \times \id_X) ({\id_X} \times {r}) (r \times \id_X) = ({\id_X} \times r) (r \times \id_X) ({\id_X} \times r)
\end{align*}
as maps $X^3 \to X^3$. Associated to this data is the \emph{structure monoid}
\begin{align*}
M(X,r) \coloneqq \langle\; X \mid xy=uv \text{ whenever } r(x,y) = (u,v) \,\rangle_{\text{monoid}}.
\end{align*}
Gateva-Ivanova--Van den Bergh proved the following result, in which undefined terms are explained in Section \ref{sec:preliminaries}.

\begin{theorem}[{\cite[Theorem 1.5]{gateva}}] \label{thm:Gateva}
Let $X$ be a finite set of size $n$ and let $r : X^2 \to X^2$ be a solution to the Yang--Baxter equation which is involutive and nondegenerate. Suppose also that $\alpha : M(X,r) \times M(X,r) \to \Bbbk^\times$ is a monoid $2$-cocycle. Then the twisted semigroup algebra $A = \Bbbk^\alpha M(X,r)$ has the following properties:
\begin{enumerate}[{\normalfont (1)},topsep=0pt,itemsep=0pt,leftmargin=35pt]
\item $A$ is a noetherian domain;
\item $A$ is AS regular of global dimension $n$;
\item $A$ has Hilbert series $1/(1-t)^n$;
\item $A$ is Koszul; and
\item $A$ satisfies the Auslander condition and is GK-Cohen--Macaulay.
\end{enumerate}
\end{theorem}

We refer to the algebras $\Bbbk^\alpha M$ as \emph{twisted structure algebras}, or simply \emph{structure algebras} when $\alpha$ is trivial. One can show (see Examples \ref{ex:Goetz1} and \ref{ex:Goetz2}) that the algebras studied in \cite{goetz} are (twisted) structure algebras for suitably chosen solutions to the Yang--Baxter equation, and hence have good ring-theoretic and homological properties.

In the years since the publication of \cite{gateva}, there has been a sustained interest in the Yang--Baxter equation and its relationship with semigroups, braces, cycle sets, and other algebraic structures; see \cite{rump2005,rump,chouraqui,dehornoy,gatevamultipermutation}. While there has been some focus on noncommutative rings arising from solutions to the Yang--Baxter equation, such as in \cite{gatevaquadratic}, there appears to have been less work from the perspective of noncommutative invariant theory and related areas.

\subsection{Main results}

For the remainder of the introduction, we write $X$ for a finite set of size $n$ and $r: X^2 \to X^2$ for a nondegenerate, involutive solution to the Yang--Baxter equation. We shall assume $r$ has Dehornoy class $c$; see Definition \ref{defn:DehornoyClass}. We also write $M(X,r)$ for the structure monoid associated to $r$, and let $\alpha : M(X,r) \times M(X,r) \to \Bbbk^\times$ be a normalised monoid $2$-cocycle.

Our first main result associates to $r : X^2 \to X^2$ an infinite family of groups which are dual reflection groups for the corresponding (twisted) structure algebra. This generalises results from \cite{goetz,buzaglo}. We stress, however, that much of the hard work in the proof of this result is due to \cite{dehornoy}.

\begin{theorem}[{Theorem \ref{thm:DualReflection}}] \label{thm:IntroThm1}
For each $k \geqslant 1$, to $r : X^2 \to X^2$ one can associate a certain finite group $G \coloneqq G(X,r,k)$ of order $(ck)^n$ which is a quotient of $M \coloneqq M(X,r)$. This endows $\Bbbk^\alpha M$ with a natural $G$-grading. The corresponding subring of dual invariants $(\Bbbk^\alpha M)_{1_G}$ is a quantum polynomial ring in $n$ variables; that is, $G$ is dual reflection group for $\Bbbk^\alpha M$.
\end{theorem}

If $\alpha$ is trivial or, more generally, the pullback along the projection $\pi_k : M(X,r) \to G_k$ of some group $2$-cocycle $\overline{\alpha} : G_k \times G_k \to \Bbbk^\times$, then $A_{1_{G_k}}$ is a polynomial ring in $n$ variables.

While all of the relations in a structure algebra are binomial, performing calculations with elements can be quite tedious. For example, it can be highly nontrivial to determine the centres of these algebras. However, due to the additional structure on $\Bbbk M$ afforded by the underlying solution to the Yang--Baxter equation, we are able to describe the centre very explicitly, at least in the case where the cocycle twist is trivial.

Before stating our result, we introduce some notation. By \cite{rump2005}, if $r : X^2 \to X^2$ is a nondegenerate, involutive solution to the Yang--Baxter equation, then there exist permutations $\psi_a : X \to X$ for each $a \in \NN^X$ which satisfy
\begin{align*}
r(x,\psi_x(y)) = (y, \psi_y(x))
\end{align*}
for all $x,y \in X$.

\begin{theorem}[{Theorem \ref{thm:Centre}}]
With notation as above, write $H = \langle \psi_x \mid x \in X \rangle \leqslant \Sym(X)$ and $K = \{ a \in \NN^X \mid \psi_a = \id_X \}$. Also write $A = \Bbbk M$ and $R = \Bbbk[K] \subseteq \Bbbk [X]$. There is an isomorphism
\begin{align*}
Z(A) \cong R^H.
\end{align*}
\end{theorem}

The above result is practical, in the sense that it can be used to determine generators for the centre and relations between these generators; see Example \ref{ex:BuzagloCentre}.

By Theorem \ref{thm:Gateva}, (twisted) structure algebras are known to be Koszul, which means that the trivial module has a resolution where each map has degree $1$. We show that the Koszul complex for a structure algebra can be described combinatorially.

\begin{theorem}[{Proposition \ref{prop:KoszulComplex}}]
Write $X = \{x_1, \dots, x_n\}$. For $0 \leqslant m \leqslant n$, let $V_m$ be the vector space with basis $\{ v_I \mid I = \{ i_1, \dots, i_m\} \subseteq \{1, \dots, n\} \}$. The Koszul complex for $A = \Bbbk M$ has differentials
\begin{gather*}
\partial_m : V_m \otimes A \to V_{m-1} \otimes A, \qquad \partial_m(v_I \otimes a) = \sum_{j=1}^m (-1)^{m-j} v_{I\setminus\{i_j\}} \otimes \psi_{I\setminus\{i_j\}} (x_{i_j}) \, a,
\end{gather*}
where $\psi_{J}$ is a certain permutation of $X$ which depends on $J \subseteq \{ 1, \dots, n\}$.
\end{theorem}

We also give an expression for the differential when $A=\Bbbk^\alpha M$ is an arbitrary $2$-cocycle twist of $\Bbbk M$, although the formula is more unwieldy.

By work of Dubois-Violette \cite{duboisviolette}, it is known that the relations in a Koszul algebra can be obtained from a so-called twisted superpotential. For twisted structure algebras, the twisted superpotential has a particularly nice description in terms of factorisations of a Garside element $\Delta$ in $M$.

\begin{theorem}[{Proposition \ref{prop:Superpotential}}]
Let $\Delta$ be the Garside element of minimal length in the structure monoid $M(X,r)$. The twisted superpotential associated to $\Bbbk^\alpha M(X,r)$ is of the form
\begin{align*}
\sfw_\alpha = \sum_{\substack{y_1, \dots, y_n \in X \\ y_1 \dots y_n = \Delta}} \lambda_{y_1, \dots, y_n}^\alpha \; y_1 \otimes \dots \otimes y_n,
\end{align*}
for some $\lambda_{y_1, \dots, y_n}^\alpha \in \Bbbk$.
\end{theorem}

Using a result of Mori--Smith, we are able to give necessary and sufficient conditions for when a twisted structure algebra is Calabi--Yau, which are easy to check in examples (see Example \ref{ex:BuzagloCY}). 

\begin{theorem}[{Proposition \ref{prop:CY}}]
The twisted structure algebra $\Bbbk^\alpha M$ is Calabi--Yau if and only if 
\begin{align*}
\Delta \in Z(M) \quad \text{and} \quad \sgn \psi_x = \frac{\alpha(x,\Delta)}{\alpha(\Delta,x)} \quad \text{for all } x \in X.
\end{align*}
\end{theorem}

We end the paper by analysing a number of examples in the recent literature on dual reflection groups and demonstrate that they fall under the purview of the results in this paper.

\mbox{}

\noindent\textbf{Acknowledgements.} The author is a Heilbronn fellow at the University of Manchester. He is grateful to Lucas Buzaglo, Ellen Kirkman, and Daniel Rogalski for helpful discussions.

\newpage

\section{Preliminaries} \label{sec:preliminaries}

Throughout, let $\Bbbk$ be an algebraically closed field of characteristic $0$.

\subsection{Artin--Schelter regular algebras and Koszul algebras}
Suppose that $A$ is a $\Bbbk$-algebra. We say that $A$ is \emph{connected graded} if it is $\NN$-graded with $A_0 = \Bbbk$. A graded $A$-module $M$ is said to be \emph{locally finite} if $\dim_\Bbbk M_i < \infty$ for all $i \in \mathbb{Z}$. \\
\indent If $A$ is a connected graded $\Bbbk$-algebra which is locally finite, then we may define the \emph{Gelfand-Kirillov (GK) dimension} of a finitely generated graded $A$-module $M$ to be
\begin{align*}
\GKdim M \coloneqq \limsup_{n \to \infty} \log_n(\dim_\Bbbk M_n) + 1.
\end{align*}
Note that the finite generation hypothesis ensures that $M_n = 0$ for $n \ll 0$. The GK dimension can be defined in wider generality but, under our assumptions, the above definition is equivalent to the usual one by \cite[Proposition 6.6]{lenagan}. Under these hypotheses, we may also define the \emph{Hilbert series} of $M$ as follows:
\begin{align*}
\hilb_M(t) = \sum_{n \geqslant 0} (\dim_\Bbbk M_n)\,t^n.
\end{align*}

We will be interested in a particular family of connected graded algebras, which are often thought of as noncommutative analogues of polynomial rings:

\begin{defn}
Let $A$ be a connected graded $\Bbbk$-algebra, and also write $\Bbbk = A/A_{\geqslant 1}$ for the trivial module. We say that $A$ is \emph{Artin--Schelter Gorenstein} (or AS Gorenstein) \emph{of dimension $d$} if:
\begin{enumerate}[{\normalfont (1)},topsep=1pt,itemsep=1pt,leftmargin=35pt]
\item $A$ has injective dimension $d$ as a left or right module over $A$; and
\item As graded right $A$-modules,
\begin{align*}
\Ext^i(_{A} \Bbbk,{}_A A) \cong \left \{
\begin{array}{cl}
0 & \text{if } i \neq d, \\
\Bbbk[\ell]_A & \text{if } i = d,
\end{array}
\right . 
\end{align*}
for some integer $\ell$, and a symmetric condition holds for $\Ext^i(\Bbbk_{A},A_A)$, with the same integer $\ell$. We call $\ell$ the \emph{Gorenstein parameter} of $A$.
\end{enumerate}
If, moreover
\begin{enumerate}[{\normalfont (1)},topsep=1pt,itemsep=1pt,leftmargin=35pt]
\item[(3)] $\gldim A = d$, and
\item[(4)] $A$ has finite GK dimension,
\end{enumerate}
then we say that $A$ is \emph{Artin--Schelter regular} (or AS regular) \emph{of dimension $d$}.
\end{defn}

AS regular algebras are often viewed as noncommutative analogues of polynomial rings and have good ring-theoretic and homological properties. For example, all known examples are noetherian domains, and it is conjectured that this is always the case. In particular, the only examples of commutative AS regular algebras are polynomial rings, \cite[Exercise 2.2.25]{brunsherzog}. We note that some authors omit the finite GK dimension hypothesis, in which case $A$ need not be noetherian, \cite[Theorem 0.2]{zhangnoetherian}.

A $\Bbbk$-algebra $A$ is said to be \emph{quadratic} if it has a presentation of the form
\begin{align*}
A = \frac{T_\Bbbk(V)}{\langle \mathcal{R} \rangle}, 
\end{align*}
where $V = A_1$ is a finite-dimensional vector space and $\mathcal{R} \subseteq V^{\otimes 2}$ is a subspace of relations. Since $V$ is finite-dimensional, we have an identification
\begin{align*}
\varphi : (V^*)^{\otimes k} \to (V^{\otimes k})^*, \quad \varphi(f_1 \otimes \dots \otimes f_k)(v_1 \otimes \dots \otimes v_k) = f_k(v_1) f_{k-1}(v_2) \dots f_1 (v_k). 
\end{align*}
We will usually suppress the map $\varphi$. Note that there are competing conventions (see, for example, \cite[Footnote 3]{morismith}) when identifying $(V^*)^{\otimes k}$ with $(V^{\otimes k})^*$. With this identification, we set 
\begin{align*}
\mathcal{R}^\perp = \{ f \in (V^*)^{\otimes 2} \mid f(\mathcal{R}) = 0 \},
\end{align*}
and then define the \emph{Koszul dual} of $A$ to be
\begin{align*}
A^! \coloneqq \frac{T_\Bbbk(V^*)}{\langle \mathcal{R}^\perp \rangle}.
\end{align*}
This algebra is also connected graded. As the name implies, $A^{!!} \cong A$.

A \emph{linear resolution} of a graded $A$-module is a free resolution of the form
\begin{align*}
\dots \to F_{i+1} \to F_{i} \to F_{i-1} \to \dots \to F_{1} \to F_{0} \to M \to 0,
\end{align*}
where each $F_i$ is a free $A$-module generated in degree $i$.

\begin{defn}
A \emph{Koszul algebra} is a quadratic algebra $A$ for which $\Bbbk \coloneqq A/A_{\geqslant 1}$ has a linear resolution.
\end{defn}

Now fix a basis $\{ v_1, \dots, v_n \}$ of $V=A_1$ and let $\{ f_1, \dots, f_n\}$ be the corresponding dual basis of $V^* = A_1^!$. Then the element $e = \sum_{j=1}^n f_j \otimes v_j$ is independent of our choice of basis. The vector space $F_i \coloneqq (A_i^!)^* \otimes_\Bbbk A$ is a free right $A$-module and there are $A$-module homomorphisms
\begin{align*}
\partial_{i} : F_{i} \to F_{i-1}, \quad \partial_{i}(\phi \otimes a) = \sum_{j=1}^n \phi f_j \otimes v_j a, \quad \text{where} \quad \phi f_j : A^!_{i-1} \to \Bbbk, \quad (\phi f_j)(f') = \phi(f_j f').
\end{align*}
The \emph{Koszul complex} is then given by
\begin{align*}
K_\bullet = \quad \dots \xrightarrow{\mathmakebox[22pt]{\partial_{i+1}}} F_{i} \xrightarrow{\mathmakebox[22pt]{\partial_{i}}} F_{i-1} \xrightarrow{\mathmakebox[22pt]{\partial_{i-1}}} \dots \xrightarrow{\mathmakebox[22pt]{\partial_{2}}} F_1 \xrightarrow{\mathmakebox[22pt]{\partial_{1}}} F_{0} \xrightarrow{\mathmakebox[22pt]{\varepsilon}} \Bbbk \xrightarrow{\makebox[22pt]{\phantom{a}}} 0
\end{align*}
which, as the name suggests, is a complex of $A$-modules. 

\begin{proposition}[{\cite[Theorem 5.9]{smithfindim}}]
A quadratic algebra $A$ is Koszul if and only if its Koszul complex is exact. This happens if and only if $A^!$ is Koszul. In this case, the Hilbert series of $A$ and $A^!$ satisfy
\begin{align*}
\hilb_{A^!}(t) = \frac{1}{\hilb_A(-t)}. 
\end{align*}
\end{proposition}

Koszul algebras have been well studied and many familiar algebras are Koszul. For example, if $A$ is AS regular of global dimension $n$ with Hilbert series $\hilb_A(t) = 1/(1-t)^n$, then it is Koszul by \cite[Proposition 3.1 (4)-(5)]{stephenson}. In particular, polynomial rings are Koszul, and hence so too are their duals, namely exterior algebras.

Closely related to the types of algebra already discussed are \emph{superpotential algebras}. Let $V$ be a finite-dimensional $\Bbbk$-vector space. For any positive integer $m$, there is a linear map 
\begin{align*}
\phi: V^{\otimes m} \to V^{\otimes m}, \qquad \phi( v_1 \otimes v_2 \otimes \dots \otimes v_{m-1} \otimes v_{m}) = v_2 \otimes v_3 \otimes \dots \otimes v_{m} \otimes v_{1}.
\end{align*}

\begin{defn} 
Let $\sfw \in V^{\otimes m}$ and let $\theta \in \GL(V)$. We say that $\sfw$ is a $\theta$\emph{-twisted superpotential} (or simply a \emph{twisted superpotential}) if $({\id^{\otimes (m-1)}} \otimes \sigma) \phi(\sfw) = \sfw$. When $\theta = \id_V$, we call $\sfw$ a \emph{superpotential}.
\end{defn}

We caution the reader that other authors have slightly different conventions for twisted superpotentials; cf.~\cite{bsw,morismith}.

\begin{defn}
Let $\sfw \in V^{\otimes m}$. We define
\begin{align*}
\partial^i \sfw \coloneqq \big\{ \psi_1 \otimes \psi_2 \otimes \dots \psi_i \otimes \id^{m-i}(\sfw) \hspace{3pt} \big| \hspace{3pt} \psi_1, \dots, \psi_i \in V^* \big \}.
\end{align*}
We then define the \emph{superpotential algebra of $\sfw$ of order $i$} to be
\begin{align*}
\mathscr{D}(\sfw,i) \coloneqq \frac{T_{\Bbbk}(V)}{\langle \partial^i \sfw \rangle}.
\end{align*}
\end{defn}

In practice, the relations in $\mathscr{D}(\sfw,i)$ are obtained from $\sfw$ by ``formal left differentiation'' with respect to all length $i$ expressions in the $v_j$. 

The following result establishes the connection between Koszul AS regular algebras and superpotential algebras:

\begin{theorem}[{\cite[Theorem 11]{duboisviolette}}]
Suppose that $A = T(V)/\langle \mathcal{R} \rangle$ is a Koszul AS regular algebra of global dimension $n$. Then there exists a twisted superpotential $\sfw \in V^{\otimes n}$ such that
\begin{align*}
A \cong \mathscr{D}(\sfw, n-2).
\end{align*}
\end{theorem}

\subsection{Group coactions}
Let $G$ be a finite group. We say that an algebra $A$ is \emph{graded by the group $G$}, or \emph{$G$-graded}, if there is a decomposition into abelian groups
\begin{align*}
A = \bigoplus_{g \in G} A_g
\end{align*}
which satisfies $A_g A_h \subseteq A_{gh}$. Necessarily $1_A \in A_{1_G}$. An element in $A_g$ for some $g \in G$ is called $G$-homogeneous. The component $A_{1_G}$, corresponding to the identity in $G$, is necessarily a subring of $A$, which we call the \emph{subring of dual invariants}. 

It is well known that a group grading on a ring $A$ is equivalent to an action of the Hopf algebra $H = (\Bbbk G)^*$ on $A$ or, equivalently, a coaction of $\Bbbk G$. Moreover, the subring of dual invariants $A_{1_G}$ is equal to the ring of invariants $A^H$, and it is possible to use various tools from noncommutative ring theory to study these rings. We will use the terms group grading and group coaction interchangeably.

Typically, one is interested in the invariant theory of finite groups when the action is faithful. The analogue of this notion for an action by a Hopf algebra is \emph{inner-faithfulness}. In the context of group gradings, the corresponding action is inner-faithful when the set
\begin{align*}
\{ \deg_G a \mid a \in A_1 \text{ is } G\text{-homogeneous}\}
\end{align*}
generates $G$, \cite[Lemma 2.5]{coactions}. All group gradings considered in this paper will be inner-faithful by construction.

As mentioned in the introduction, we are particularly interested in inner-faithful group gradings of an algebra $A$ for which the subring of dual invariants has finite global dimension. These have been studied in \cite{kirkmannakayama,goetz} and, in light of the Chevalley--Shephard--Todd Theorem, should be thought of as quantum analogues of subgroups of $\GL(n,\Bbbk)$ which are generated by reflections.

\subsection{The Yang--Baxter Equation, Cycle Sets, and Structure Monoids}
The \emph{Yang--Baxter equation} is an equation which originally arose in the study of statistical mechanics, where its formulation is slightly different from the version we introduce. Our particular interest will be in \emph{set-theoretic solutions} to the Yang--Baxter equation, whose study has been motivated by quantum groups and related structures.

\begin{defn} \label{defn:Solution}
Let $X$ be a set and $r : X^2 \to X^2$. We say that $r$ is a \emph{(set-theoretic) solution to the Yang--Baxter equation} if 
\begin{align*}
(r \times \id_X) ({\id_X} \times {r}) (r \times \id_X) = ({\id_X} \times r) (r \times \id_X) ({\id_X} \times r)
\end{align*}
as maps $X^3 \to X^3$. We say that the solution $r$ is \emph{involutive} if $r^2 = \id_{X^2}$.

Finally, suppose we write a solution $r$ as
\begin{align*}
r(x,y) = (\lambda_x(y), \rho_y(x))
\end{align*}
where $\lambda_x, \rho_y : X \to X$. We say that $r$ is \emph{nondegenerate} if $\lambda_x$ and $\rho_y$ are bijective for all $x,y \in X$.
\end{defn}

When $X$ is finite and $r$ is involutive, $\lambda_x$ is a bijection for all $x \in X$ if and only if $\rho_y$ is a bijection for all $y \in X$, by \cite[Theorem 2]{rump2005}, and so nondegeneracy is equivalent to bijectivity of $\lambda_x$ for all $x \in X$. We will always assume $X$ to be finite.

Solutions to the Yang--Baxter equation are closely related to cycle sets, whose definition we give below.

\begin{defn}
A \emph{left cycle set} is a set $X$ equipped with a binary operation $* : X \times X \to X$ such that:
\begin{enumerate}[{\normalfont (1)},topsep=1pt,itemsep=1pt,leftmargin=35pt]
\item The map
\begin{align*}
L_x : X \to X, \quad L_x(y) = x * y
\end{align*}
is bijective for all $x \in X$; and
\item The following identity holds for all $x,y,z \in X$:
\begin{align}
(x * y) * (x * z) = (y*x)*(y*z). \label{eqn:CycleSetIdentity}
\end{align}
Equivalently, for all $x,y \in X$ we have
\begin{align}
L_{x*y} \circ L_x = L_{y*x} \circ L_y. \label{eqn:CycleSetIdentity2}
\end{align}
\end{enumerate}
There is an analogous definition of a \emph{right cycle set}. We will work only with left cycle sets, which we shall simply call \emph{cycle sets}.
\end{defn}

Rump showed that nondegenerate, involutive solutions to the Yang--Baxter equation are in bijective correspondence with cycle sets.

\begin{proposition}[{\cite[Proposition 1]{rump2005}}]
There is a bijective correspondence between nondegenerate, involutive solutions to the Yang--Baxter equation and left (or right) cycle sets. More precisely, if $r(x,y) = (\lambda_x(y), \rho_y(x))$ is such a solution, then $X$ is a cycle set with $x * y = \lambda_x^{-1}(y)$. Conversely, if $(X,*)$ is a cycle set, then the map
\begin{align*}
r : X^2 \to X^2, \quad r(x,y) = \big( L_x^{-1}(y), L_{L_x^{-1}(y)}(x) \big) 
\end{align*}
is a nondegenerate, involutive solution to the Yang--Baxter equation.
\end{proposition}

Suppose $r : X^2 \to X^2$ is a solution to the Yang--Baxter equation. The \emph{structure monoid} associated to this solution is the monoid
\begin{align*}
M(X,r) \coloneqq \langle\; X \mid xy=uv \text{ whenever } r(x,y) = (u,v) \,\rangle_{\text{monoid}}.
\end{align*}
When $r$ is involutive and nondegenerate, one can show that $r$ has $n$ fixed points, so $M(X,r)$ has $\binom{n}{2}$ defining relations. Under the correspondence with cycle sets, one may present the structure monoid as
\begin{align}
M(X,r) \coloneqq \langle\; X \mid x(x*y) = y(y*x) \text{ where } x,y \in X \,\rangle_{\text{monoid}}. \label{eqn:StructureMonoidAlt}
\end{align}

Since the relations in $M(X,r)$ involve terms of the same length, it has a well-defined length function. We will write $\ell(m)$ for the length of $m \in M(X,r)$.

If $r$ is a nondegenerate, involutive solution to the Yang--Baxter equation, then the structure monoid has a number of good properties, of which we recall a handful from \cite[Propositions 3.1 and 3.5]{dehornoy}.

\begin{proposition}
The structure monoid $M(X,r)$ has the following properties:
\begin{enumerate}[{\normalfont (1)},topsep=0pt,itemsep=0pt,wide=0pt]
\item it is left- and right-cancellative;
\item it has unique left- and right-lcms and unique left- and right-gcds; and
\item it has a \emph{Garside element} $\Delta \in M$, an element for which the set of left- and right-divisors coincide, are finite in number, and generate $M$. In $M(X,r)$, $\Delta$ has exactly $2^n$ divisors.
\end{enumerate}
\end{proposition}

Suppose that $M$ is a monoid and that $\alpha : M \times M \to \Bbbk^\times$ is a monoid $2$-cocycle, i.e.~a function which satisfies
\begin{align*}
\alpha(a,b) \alpha(ab,c) = \alpha(a,bc) \alpha(b,c)
\end{align*}
for all $a,b,c \in M$. We say that $\alpha$ is \emph{normalised} if $\alpha(a,1) = 1 = \alpha(1,a)$ for all $a \in M$. Without loss of generality, we may assume that $\alpha$ is normalised, and we will make this assumption throughout. The \emph{twisted semigroup algebra} $\Bbbk^\alpha M$ is the unital $\Bbbk$-algebra whose $\Bbbk$-vector space structure coincides with that of the semigroup algebra $\Bbbk M$, but with multiplication given by
\begin{align*}
a \star b \coloneqq \alpha(a,b) ab.
\end{align*}
It is straightforward to show that a presentation for $\Bbbk^\alpha M(X,r)$ is
\begin{align*}
\Bbbk^\alpha M(X,r) \cong \frac{\Bbbk \langle X \rangle}{\big\langle xy - \frac{\alpha(x,y)}{\alpha(u,v)} uv \mid x,y,u,v \in X \text{ and } r(x,y) = (u,v) \big\rangle}.
\end{align*}

For the structure monoid $M(X,r)$ we shall refer to the semigroup algebra $\Bbbk M(X,r)$ as the \emph{structure algebra} associated to the solution $r : X^2 \to X^2$. If $\alpha$ is a $2$-cocycle on $M$, then we call $\Bbbk^\alpha M(X,r)$ the \emph{twisted structure algebra} associated to $r$.

\subsection{\texorpdfstring{$I$}{I}-structures}
By \cite{gateva,dehornoy}, the structure monoid $M(X,r)$ associated to a solution to the Yang--Baxter equation can be equipped with an \emph{$I$-structure}. This was the main tool used to prove the result of Gateva-Ivanova--Van den Bergh given in Theorem \ref{thm:Gateva}.

\begin{defn}
Let $M$ be a monoid with generating set $X$ and let $\NN^X$ be the free abelian monoid with generating set $X$, written multiplicatively. A \emph{(right) $I$-structure} on $M$ is a bijective map $\nu : \NN^X \to M$ satisfying $\nu(1) = 1$, $\nu(x) = x$ for all $x \in X$, and
\begin{align}
\{ \nu(ax) \mid x \in X \} = \{ \nu(a) x \mid x \in X \} \quad \text{for all } x \in X. \label{eqn:IStructure}
\end{align}
In particular, \eqref{eqn:IStructure} is equivalent to the existence of, for each $a \in \NN^X$, a permutation $\psi_a \in \Sym(X)$ satisfying
\begin{align*}
\nu(ax) = \nu(a) \psi_a(x) \quad \text{for all } x \in X.
\end{align*}
\end{defn}

We introduce some notation, following \cite{dehornoy}.

\begin{defn}
Let $X$ be a cycle set. For $n \geqslant 1$, inductively define maps $\Omega_n : X^n \to X$ by $\Omega_1(x_1) = x_1$ and
\begin{align*}
\Omega_n(x_1, \dots, x_n) = \Omega_{n-1}(x_1, \dots, x_{n-1}) * \Omega_{n-1}(x_1, \dots, x_{n-2}, x_n).
\end{align*}
Furthermore, for $n \geqslant 1$, define maps $\Pi_n : X^n \to X^n$ by
\begin{align*}
\Pi_n(x_1, \dots, x_n) = \big(\Omega_1(x_1), \Omega_2(x_1, x_2), \dots, \Omega_n(x_1, \dots, x_n)\big).
\end{align*}
\end{defn}

We caution the reader that our definition of $\Pi_n$ is slightly different from the one given by Dehornoy: he defines it to be the formal expression
\begin{align*}
\Omega_1(x_1) \cdot \Omega_2(x_1, x_2) \cdot \ldots \cdot \Omega_n(x_1, \dots, x_n)\big).
\end{align*}

Using these definitions, Dehornoy defines what it means for a solution to the Yang--Baxter equation to have \emph{class} $d$. 

\begin{defn} \label{defn:DehornoyClass}
Let $r : X^2 \to X^2$ be a nondegenerate, involutive solution to the Yang--Baxter equation. We say that $r$ has \emph{Dehornoy class} $d$ if
\begin{align*}
\Omega_{d+1}(x, \dots, x,y) = y \quad \text{for all } x,y \in X.
\end{align*}
If $c$ is minimal amongst all such $d$, then we call it the \emph{minimal Dehornoy class}.
\end{defn}

By \cite[Lemma 5.4]{dehornoy}, every such solution is of Dehornoy class $d$ for some $d < \infty$. The proof also shows that if $c \geqslant 1$ is the minimal Dehornoy class and $r$ is also of class $d$, then $c \mid d$.

Dehornoy also proves the following in \cite[Lemmas 2.2 and 2.6]{dehornoy}:

\begin{lemma} \label{lem:OmegaAndPiProp}
\begin{enumerate}[{\normalfont (1)},topsep=0pt,itemsep=0pt,wide=0pt]
\item For all $\pi \in S_{n-1}$ and $x_1, \dots, x_n \in X$, the map $\Omega_n$ satisfies
\begin{align*}
\Omega_n(x_{\pi(1)}, \dots, x_{\pi(n-1)}, x_n) = \Omega_n(x_1, \dots, x_n).
\end{align*}
\item The map $\Pi_n$ is a symmetric function on $X^n$.
\end{enumerate}
\end{lemma}
 
Setting $n=3$ in part (1) recovers \eqref{eqn:CycleSetIdentity}, and so this is a natural extension of that property to larger sets.

Dehornoy then goes on to show that we can use these maps to define an $I$-structure on the structure monoid $M(X,r)$:

\begin{proposition}[{\cite[Proposition 4.2]{dehornoy}}] \label{prop:IStructure}
Let $\mu_n : X^n \to M(X,r)$ be the multiplication map. The map
\begin{align*}
\nu : \NN^X \to M(X,r), \quad \nu_n(x_1 \dots x_n) = \mu \circ \Pi_n(x_1, \dots, x_n)
\end{align*}
defines an $I$-structure on $M(X,r)$.
\end{proposition}

Dehornoy proves a number of useful properties about the map $\nu$. We list below those properties which we will have need for in the present article.

\begin{lemma}[{\cite[Lemma 4.5]{dehornoy}}] \label{lem:PsiToProduct}
If $a, b \in \NN^X$ then
\begin{align*}
\nu(ab) =\nu(a) \nu(\psi_a (b)) \quad \text{and} \quad \psi_{ab} = \psi_{\psi_a(b)} \circ \psi_a,
\end{align*}
where we abuse notation and write $\psi_a(b)$ for the result of applying $\psi_a$ to $b$ componentwise.
\end{lemma}

Using $\nu$, it is easy to determine a Garside element $\Delta$ for $M$, as well as some of its properties. Henceforth, we write $\Delta$ for the Garside element of minimal length $n$. We will speak of this element as being \emph{the} Garside element in $M$.

\begin{lemma} \label{lem:PropertiesOfDelta}
\begin{enumerate}[{\normalfont (1)},topsep=0pt,itemsep=0pt,wide=0pt]
\item Let $\delta = \prod_{x \in X} x \in \NN^X$. Then the Garside element of $M(X,r)$ is given by $\Delta = \nu(\delta)$.
\item There is a bijective correspondence between factorisations $\Delta = y_1 \dots y_n$ in $M$ with each $y_i \in X$ and permutations of $X$.
\item The element $\Delta$ satisfies the following property:
\begin{align*}
x \Delta = \Delta \psi_\delta(x). 
\end{align*}
\end{enumerate}
\end{lemma}
\begin{proof}
\indent (1) This is \cite[Lemma 5.6]{dehornoy}. \\
\indent (2) We first show that each permutation gives rise to a factorisation of $\Delta$. Write $X = \{x_1, \dots, x_n\}$ and let $\sigma \in S_n$. Also write Write $\Pi_n(x_{\sigma(1)}, \dots, x_{\sigma(n)}) = (y_1, \dots, y_n)$. Then, by \cite[Lemma 2.6]{dehornoy},
\begin{align*}
y_1 \dots y_n = \mu_n\big(\Pi_n(x_{\sigma(1)}, \dots, x_{\sigma(n)})\big) = \nu(x_{\sigma(1)} \dots x_{\sigma(n)}) = \nu(\delta) = \Delta.
\end{align*}
Moreover, distinct permutations give rise to distinct factorisations since $\Pi_n$ is bijective by \cite[Lemma 2.3]{dehornoy}. 

Conversely, suppose $\Delta = y_1 \dots y_n$ is a factorisation of $\Delta$ with $y_i \in X$. Since $\Pi_n$ is surjective, there is a unique $(s_1, \dots, s_n) \in X^n$ with $\Pi_n(s_1, \dots, s_n) = (y_1, \dots, y_n)$. Then
\begin{align*}
\nu(s_1 \dots s_n) = \mu_n \Pi_n(x_1, \dots, x_n) = y_1 \dots y_n = \Delta = \nu(\delta),
\end{align*}
so injectivity of $\nu$ gives $\delta = s_1 \dots s_n$ in $\NN^X$. Thus $(s_1, \dots, s_n)$ is permutation of the elements of $X$, and this permutation is uniquely determined.

(3) Using Lemma \ref{lem:PsiToProduct}, we find
\begin{gather*}
x \Delta = \nu(x) \nu(\delta) = \nu(x \psi_{x}^{-1}(\delta)) = \nu(x\delta) = \nu(\delta x) = \nu(\delta) \psi_\delta(x) = \Delta \psi_\delta(x). \qedhere
\end{gather*}
\end{proof}

Using the theory developed above, it is possible to define a useful nontrivial $2$-cocycle on any structure monoid $M(X,r)$, provided that $|X| \geqslant 2$.

\begin{example} \label{ex:NormalForm} 
Write $X = \{ x_1, \dots, x_n \}$ and define an ordering on $X$ by $x_1 \prec x_2 \prec \dots \prec x_n$. Write $X^* = \bigcup_{m \in \NN} X^m$ and define the \emph{inversion number} of $S = (s_1, \dots, s_m) \in X^*$ by
\begin{align*}
\inv(S) \coloneqq |\{(i,j) \mid i>j \text{ and } s_i \prec s_j \}|.
\end{align*}
We use this to define a normal form on elements of $M$. Given $a \in M$ of length $m$, we define the normal form of $a$, denoted $\NF(a) \in X^m$, to be the uniquely determined $\hat{a} \in X^m$ with $\mu_m(\hat{a}) = a$ and $\inv(\Pi_m^{-1}(\hat{a}))=0$. We also write 
\begin{align*}
\varepsilon(a) = (-1)^{\inv(\Pi_m^{-1}(a))}.
\end{align*}

Let $s = (s_1, \dots, s_m)$ and $s' = (s_1', \dots, s_m')$ be elements of $X^m$ with $\mu_m(s) = \mu_m(s')$. The proof of \cite[Proposition 4.2]{dehornoy} shows that if the equality $\mu_m(s) = \mu_m(s')$ can be witnessed by applying exactly one of the defining relations of $M$, then $\Pi_m^{-1}(s)$ and $\Pi_m^{-1}(s')$ differ by transposing two adjacent coordinates. In this case, we have $\varepsilon(s) = -\varepsilon(s')$.

Now suppose that $s, s' \in X^m$ with $\mu_m(s) = \mu_m(s')$, and suppose that there is a sequence of elements $s_i \in X^m$,
\begin{align*}
s = s_0 \to s_1 \to s_2 \to \dots \to s_{k-1} \to s_k = s',
\end{align*}
with $\mu_m(s_i) = \mu_m(s)$ for all $i$, and the equality $\mu_m(s_i) = \mu_m(s_{i+1})$ is witnessed by applying exactly one defining relation from $M$. It follows that $\varepsilon(s) = (-1)^k \varepsilon(s')$ or, equivalently, 
\begin{align*}
\frac{\varepsilon(s)}{\varepsilon(s')} = (-1)^k.
\end{align*}
This quantity depends only on the start and end points, so the parity of the number of relations required to witness the equality $\mu_m(s) = \mu_m(s')$ is well-defined. Write $\kappa(s,s') \in \ZZ/2\ZZ$ for this quantity. 

We define our claimed $2$-cocycle by
\begin{align*}
\rho : M \times M \to \Bbbk^\times, \quad \rho(a,b) = (-1)^{\inv \big( \Pi_{\ell(ab)}^{-1}(\NF(a) \frown \NF(b)) \big)},
\end{align*}
where $\NF(a) \frown \NF(b)$ denotes the concatenation of $\NF(a)$ and $\NF(b)$. The exponent in the definition of $\rho(a,b)$ records the number of defining relations of $M$ that are required to put $ab$ into normal form, where $a$ and $b$ are assumed to be in normal form.

This is a $2$-cocycle: if $a,b,c \in M$, then the parity of the number of relations required to put $\NF(a) \NF(b) \NF(c)$ into normal form can be calculated by first putting $\NF(a) \NF(b)$ into normal form, then putting $\NF(ab) \NF(c)$ into normal form, or by first putting $\NF(b) \NF(c)$ into normal form, then putting $\NF(a) \NF(bc)$ into normal form. The first option determines $\kappa(a,b) + \kappa(ab,c)$, while the second option determines $\kappa(b,c) + \kappa(a,bc)$, and these must be the same, modulo $2$. This implies that $\rho(a,b)\rho(ab,c) = \rho(a,bc) \rho(b,c)$, so $\rho$ is a $2$-cocycle. 

It is easy to write down a presentation of the twisted structure algebra for the $2$-cocycle $\rho$. Indeed, we have
\begin{align*}
\Bbbk^\rho M(X,r) \cong \frac{\Bbbk \langle X \rangle}{\big\langle xy + uv \mid x,y,u,v \in X \text{ and } r(x,y) = (u,v) \big\rangle}.
\end{align*}
\end{example}

One of the main results of \cite{dehornoy} states that a cycle set (equivalently, a nondegenerate, involutive solution to the Yang--Baxter equation) give rise to a family of finite groups which are quotients of the structure monoid. These groups can be thought of as generalisations of Coxeter groups.

\begin{proposition}[{\cite[Proposition 5.2]{dehornoy}}] \label{prop:Coxeter}
Let $r : X^2 \to X^2$ be nondegenerate, involutive soution to the Yang--Baxter equation of Dehornoy class $d$. Then the quotient of $M(X,r)$ obtained by imposing the relations $\nu(x^d) = 1$ for all $x \in X$ is a group of order $d^{|X|}$. Moreover, the preimage of $1_G$ under the natural map $\pi : M(X,r) \to G$ is isomorphic to $\ZZ^{|X|}$.
\end{proposition}

We shall refer to the groups obtained above as \emph{finite Dehornoy quotient groups}.


\section{Finite Dehornoy quotient groups are dual reflection groups}
For the remainder of this paper, let $X$ be a finite set of size $n$ and let $r : X^2 \to X^2$ be a solution to the Yang--Baxter equation, always assumed to be nondegenerate and involutive, and let $c$ be its minimal Dehornoy class. We will write $M = M(X,r)$ for the structure monoid associated to $r$ and $\alpha : M \times M \to \Bbbk^\times$ will denote a $2$-cocycle for $M$.

The algebras of particular interest to us are the semigroup algebras $\Bbbk M$ and their cocycle twists $\Bbbk^\alpha M$. These algebras are $\NN$-graded by declaring $\deg x = 1$ for all $x \in X$. By Theorem \ref{thm:Gateva}, these algebras are AS regular and enjoy a number of good ring-theoretic and homological properties. 

We now identify a family of groups which naturally coact on $\Bbbk^\alpha M$ and which are dual reflection groups for these algebras.

Let $G_k \coloneqq G(X,r,k)$ be the quotient of $M(X,r)$ with presentation
\begin{align*}
G_k \coloneqq G(X,r,k) 
&= \left\langle X \hspace{4pt} \bigg| \hspace{-1pt}
\begin{array}{cl}
xy = uv & \text{if } r(x,y)=(u,v) \\
\nu(x^c)^k=1 & \text{for all } x \in X
\end{array} \hspace{-1pt}
\right\rangle,
\end{align*}
where we recall the definition of $\nu : \NN^X \to M(X,r)$ from Proposition \ref{prop:IStructure}. We will also write $\pi_k : M(X,r) \to G_k$ for the natural projection. The following is a slight modification of Proposition \ref{prop:Coxeter}.

\begin{lemma}
Fix $k \geq 1$. Then $G_k$ is a group of order $(ck)^n$.
\end{lemma}
\begin{proof}
Since $r$ is of minimal Dehornoy class $c$, it is of class $d = ck$ for any $k \geqslant 1$. By Proposition \ref{prop:Coxeter}, the group with presentation
\begin{align*}
\left\langle X \hspace{4pt} \bigg| \hspace{-1pt}
\begin{array}{cl}
xy = uv & \text{if } r(x,y)=(u,v) \\
\nu(x^{ck})=1 & \text{for all } x \in X
\end{array} \hspace{-1pt}
\right\rangle,
\end{align*}
has order $(ck)^n$. The only thing left to show is that $\nu(x^{ck}) = \nu(x^{c})^k$, but this follows from \cite[Lemma 5.5]{dehornoy} since
\begin{gather*}
\nu(x^{ck}) = \nu(x^c \cdot x^{c(k-1)} = \nu(x^c) \nu(x^{c(k-1)}) = \dots = \nu(x^c)^k. \qedhere
\end{gather*}
\end{proof}

In addition to its $\NN$-grading, the algebra $\Bbbk^\alpha M$ also possesses a $G_k$-grading by declaring $\deg_{G_k} x = \pi_k(x)$ for all $x \in X$, where the element $x$ on the left hand side lives inside $\Bbbk^\alpha M$, while on the right hand side it lives in $G_k$. Again, the majority of the hard work is completed in \cite{dehornoy}.

\begin{theorem} \label{thm:DualReflection}
Fix $k \geqslant 1$ and write $A \coloneqq \Bbbk^\alpha M$ and $G_k \coloneqq G(X,r,k)$, where $\alpha : M \times M \to \Bbbk$ is a normalised $2$-cocycle. If $x \in X$, write $a_x = \nu(x^c)$. Then the subring of dual invariants $A_{1_{G_k}}$ is
\begin{align*}
A_{1_{G_k}} = \Bbbk \langle\, a_x^k \mid x \in X \,\rangle
\end{align*}
which is isomorphic to a quantum polynomial ring in $n$ variables. That is, $G_k$ is a dual reflection group for $A$. Moreover, if $\alpha$ is trivial or, more generally, the pullback along the projection $\pi_k$ of some group $2$-cocycle $\overline{\alpha} : G_k \times G_k \to \Bbbk^\times$, in the sense that $\alpha = \overline{\alpha} \circ (\pi_k \times \pi_k)$, then $A_{1_{G_k}}$ is a polynomial ring in $n$ variables.
\end{theorem}
\begin{proof}
A $\Bbbk$-basis for the subring of dual invariants is given by the preimage of $1_{G_k}$ under the map $\pi_k : M \to G_k$. By Proposition \ref{prop:Coxeter}, this preimage is a free abelian group of rank $n$ with generating set $\{ \nu( x^{ck}) \mid x \in X \} = \{ a_x^k \mid x \in X \}$. Moreover, if $x,y \in X$ then in $\Bbbk^\alpha M$ we have
\begin{align*}
a_x^k * a_y^k = \alpha(a_x^k, a_y^k) a_x^k a_y^k = \alpha(a_x^k, a_y^k) a_y^k a_x^k = \frac{\alpha(a_x^k, a_y^k)}{\alpha(a_y^k, a_x^k)} a_y^k * a_x^k,
\end{align*}
so that $A_{1_{G_k}}$ has presentation
\begin{align*}
A_{1_{G_k}} \cong \frac{\Bbbk \langle a_x^k \mid x \in X\rangle}{\Big\langle a_x^k a_y^k = \frac{\alpha(a_x^k, a_y^k)}{\alpha(a_y^k, a_x^k)} a_y^k a_x^k \;\Big|\; x,y \in X \Big\rangle},
\end{align*}
which is a quantum polynomial ring in $n$ variables.

Finally, if $\alpha = \overline{\alpha} \circ (\pi_k \times \pi_k)$ for some group $2$-cocycle $\overline{\alpha} : G_k \times G_k \to \Bbbk^\times$, then
\begin{align*}
\alpha(a_x^k, a_y^k) = \overline{\alpha}\big( \pi_k(a_x^k), \pi_k(a_y^k)\big) = \overline{\alpha}(1,1) = 1,
\end{align*}
and hence $A_{1_{G_k}}$ is a commutative polynomial ring.
\end{proof}

Using the above result, we are able to determine precisely when $\Bbbk^\alpha M$ is PI. This is a refinement of \cite[Theorem 1.4 (6)]{gateva}.

\begin{proposition}
Fix $k \geqslant 1$, let $A = \Bbbk^\alpha M$, and $B = A_{1_{G_k}} = \Bbbk\langle\,b_x \mid x \in X\,\rangle$, where $b_x = \nu(x^c)^k$.
\begin{enumerate}[{\normalfont (1)},topsep=1pt,itemsep=1pt,wide=0pt]
\item $A$ is a finite module over $B$ of rank $(ck)^n$.
\item Write $q_{xy} = \frac{\alpha(b_x,b_y)}{\alpha(b_y,b_x)}$. Then $A$ is PI if and only if $q_{xy}$ is a root of unity for all $x,y \in X$.
\end{enumerate}
\end{proposition}
\begin{proof}
(1) Write $d = ck$ and let 
\begin{align*}
\mathcal{R} \coloneqq \Big\{ \prod_{x \in X} x^{r_x} \;\Big|\; 0 \leqslant r_x < d \text{ for all } x \in X \Big\},
\end{align*}
which is a set of size $d^n$. Now, if $p \in \NN^X$, then there exist unique $r \in \mathcal{R}$ and $q \in \NN^X$ with $p = rq^d$. Since $q^d \in \NN^X$, we have $ b = \psi_x(q^d) \in B$, so that, in $M$,
\begin{align*}
\nu(p) = \nu(r q^d) = \nu(r) \nu(\psi_r(q^d)) = \nu(r) b.
\end{align*}
It follows that $A = \sum_{r \in \mathcal{R}} \nu(r)B$. In fact, this sum is direct since the residue $r$ is unique and $\nu$ is bijective, so $A = \bigoplus_{r \in \mathcal{R}} \nu(r)B$ is free of rank $|\mathcal{R}| = (ck)^n$.

(2) By part (1), $A$ is a finite module over the quantum polynomial ring
\begin{align*}
B = \frac{\Bbbk \langle\, b_x \mid x \in X\,\rangle}{\big\langle b_x b_y = q_{xy} b_y b_x \;\big|\; x,y \in X \big\rangle}.
\end{align*}
Assume first that $q_{xy}$ is a root of unity for all $x,y \in X$, and let $N \geqslant 1$ be such that $q_{xy}^N = 1$ for all $x,y \in X$. Let $C = \Bbbk[\, b_x^N \mid x \in X\,]$, which is a commutative polynomial ring, and $B$ is free of rank $N^n$ over $C$. Therefore $A$ is a finite (free) module over the commutative ring $C$, hence PI. 

Conversely, assume that there exist $x,y \in X$ such that $q_{xy}$ is not a root of unity. Then the subalgebra of $A$ generated by $b_x$ and $b_y$ is a quantum plane:
\begin{align*}
\Bbbk \langle\,b_x,b_y \,\rangle = \Bbbk\langle u,v\rangle/\langle vu-q_{xy}uv \rangle.
\end{align*}
Since $q_{xy}$ is not a root of unity, it is well known that this algebra is not PI. Therefore, since $A$ has a non-PI subalgebra, it cannot be PI.
\end{proof}

\subsection{Some examples} 
Before proving further results about these algebras, it is prudent to give some examples.

\begin{example}
Let $X = \{x,y\}$ and $r : X^2 \to X^2$ be the solution to the Yang--Baxter equation given by
\begin{align*}
\begin{array}{cc}
r(x,x) = (y,y), & r(x,y) = (x,y). \\
r(y,x) = (y,x), & r(y,y) = (x,x),
\end{array}
\end{align*}
Then $M = M(X,r) = \langle x,y \mid x^2=y^2 \rangle$ and the associated semigroup algebra has presentation $A = \Bbbk\langle x,y \rangle/\langle x^2-y^2 \rangle$. We note that this algebra has no $2$-cocycle twists, so we work only with $A$ in this example.

Let us determine $G_k$ for $k \geqslant 1$. In this example, the corresponding cycle set $(X,*)$ satisfies
\begin{align*}
\begin{array}{cc}
x*x = y, & x*y = x. \\
y*x = y, & y*y = x.
\end{array}
\end{align*}
To work out the minimal Dehornoy class $c$, we calculate as follows:
\begin{gather*}
\Omega_3(x,x,x) = \Omega_2(x,x) * \Omega_2(x,x) = (x*x)*(x*x) = y*y = x, \\
\Omega_3(x,x,y) = \Omega_2(x,x) * \Omega_2(x,y) = (x*x)*(x*y) = y*x = y.
\end{gather*}
Similarly, $\Omega_3(y,y,y) = y$ and $\Omega_3(y,y,x) = x$, and the above calculations showed that, for example, $\Omega_2(x,y) \neq y$, so $c=2$.

To obtain a presentation for $G_k = G(X,r,k)$, we therefore need to adjoin the relations $\nu(x^2)^k = 1 = \nu(y^2)^k$. We have
\begin{align*}
\nu(x^2) = \Omega_1(x) \Omega_2(x,x) = x(x*x) = xy, \quad \nu(y^2) = \Omega_1(y) \Omega_2(y,y) = y(y*y) = yx,
\end{align*}
so that $G_k$ is the group of order $4k^2$ with presentation
\begin{align*}
G_k = \langle x,y \mid x^2=y^2, (xy)^k = (yx)^k = 1 \rangle.
\end{align*}
By Theorem \ref{thm:DualReflection}, the subring of dual invariants is $\Bbbk\langle (xy)^k, (yx)^k \rangle$, which is a polynomial ring, hence AS regular. This agrees with \cite[Theorem 5.5]{coactions}.
\end{example}

\begin{example} \label{ex:BuzagloExample}
In \cite[Theorem 6.3]{buzaglo}, the authors show that the groups
\begin{align*}
\Delta_n = \langle a,b,c \mid a^2=bc, \; b^2=ca, \; c^2=ab, \; (acb)^n=1 \rangle 
\end{align*}
are dual reflection groups for the Sklyanin algebras
\begin{align*}
S_{q,0,1} \coloneqq
\frac{\Bbbk \langle x,y,z \rangle}{\langle \, x^2+qyz, \; y^2+qzx, \; z^2+qxy \,\rangle},
\end{align*}
where $q \in \Bbbk^\times$. We show that this result falls under the umbrella of Theorem \ref{thm:DualReflection}.

Let $X = \{ x,y,z \}$. We can give $X$ the structure of a cycle set by declaring that $L_a = (x\;z\;y)$ for all $a \in X$, from which it is immediate that condition \eqref{eqn:CycleSetIdentity2} holds. Using \eqref{eqn:StructureMonoidAlt}, we find that the relations in the structure monoid $M$ associated to the corresponding solution $r$ to the Yang--Baxter equation are
\begin{align*}
x^2 = x(x*y) = y(y*x) = yz, \quad y^2 = y(y*z) = z(z*y) = zx, \quad z^2 = z(z*x) = x(x*z) = xy,
\end{align*}
so that
\begin{align*}
M = M(X,r) = \langle \, x,y,z \mid x^2=yz, \; y^2=xz, \; z^2=xy \, \rangle_{\text{monoid}}.
\end{align*}
It follows that $\Bbbk M = S_{-1,0,1}$.

We claim that there exists a $2$-cocycle $\alpha : M \times M \to \Bbbk$ so that $\Bbbk^\alpha M = S_{q,0,1}$. If $a \in M$, write $a_x, a_y, a_z$ for the exponents of $x,y,z$, respectively, in $\nu^{-1}(a) \in \NN^X$. Now define a bilinear form
\begin{align*}
B : \ZZ^3 \times \ZZ^3 \to \ZZ, \quad B(u,v) = u_x v_y + u_y v_z + u_z v_x.
\end{align*}
It is straightforward to verify that $B$ is invariant under $\sigma = (x\;z\;y)$, in the sense that $B(\sigma(u),\sigma(v))=B(u,v)$ for all $u,v \in \ZZ^3$. We will abuse notation and view $B$ as a map $\NN^X \times \NN^X \to \ZZ$, which is bilinear in the sense that $B(ab,c) = B(a,c) + B(b,c)$, for example.

Now define
\begin{align*}
\alpha(a,b) = (-q)^{B(\nu^{-1}(a), \sigma^{-\ell(a)} \nu^{-1}(b))}.
\end{align*}
We claim that $\alpha$ is the desired $2$-cocycle. We first note that, using Lemma \ref{lem:PsiToProduct} and noting that $\psi_a = (x\;z\;y) = \sigma$ for all $a \in X$, we have
\begin{align*}
\nu^{-1}(ab) = \nu^{-1}(a) \cdot \sigma^{-\ell(a)} \nu^{-1}(b).
\end{align*}
Using this we find that, if $a,b,c \in M$, then the exponent in $\alpha(a,b)\alpha(ab,c)$ is
\begin{align*}
B(\nu^{-1}&(a), \sigma^{-\ell(a)} \nu^{-1}(b)) + B(\nu^{-1}(ab), \sigma^{-\ell(ab)} \nu^{-1}(c)) \\
&= B(\nu^{-1}(a), \sigma^{-\ell(a)} \nu^{-1}(b)) + B(\nu^{-1}(a) \cdot \sigma^{-\ell(a)} \nu^{-1}(b), \sigma^{-\ell(ab)} \nu^{-1}(c)) \\
&= B(\nu^{-1}(a), \sigma^{-\ell(a)} \nu^{-1}(b)) + B(\nu^{-1}(a), \sigma^{-\ell(ab)} \nu^{-1}(c)) + B(\sigma^{-\ell(a)} \nu^{-1}(b), \sigma^{-\ell(ab)} \nu^{-1}(c)) \\
&= B(\nu^{-1}(a), \sigma^{-\ell(a)} \nu^{-1}(b)) + B(\nu^{-1}(a), \sigma^{-\ell(ab)} \nu^{-1}(c)) + B(\nu^{-1}(b), \sigma^{-\ell(b)} \nu^{-1}(c)).
\end{align*}
On the other hand, the exponent in $\alpha(a,bc)\alpha(b,c)$ is
\begin{align*}
B(\nu^{-1}&(a), \sigma^{-\ell(a)} \nu^{-1}(bc)) + B(\nu^{-1}(b), \sigma^{-\ell(b)} \nu^{-1}(c)) \\
&= B(\nu^{-1}(a), \sigma^{-\ell(a)} (\nu^{-1}(b) \cdot \sigma^{-\ell(b)} \nu^{-1}(c))) + B(\nu^{-1}(b), \sigma^{-\ell(b)} \nu^{-1}(c)) \\
&= B(\nu^{-1}(a), \sigma^{-\ell(a)} \nu^{-1}(b)) + B(\nu^{-1}(a), \sigma^{-\ell(a)-\ell(b)} \nu^{-1}(c)) + B(\nu^{-1}(b), \sigma^{-\ell(b)} \nu^{-1}(c)) \\
&= B(\nu^{-1}(a), \sigma^{-\ell(a)} \nu^{-1}(b)) + B(\nu^{-1}(a), \sigma^{-\ell(ab)} \nu^{-1}(c)) + B(\nu^{-1}(b), \sigma^{-\ell(b)} \nu^{-1}(c))
\end{align*}
Therefore $\alpha(a,b)\alpha(ab,c) = \alpha(a,bc)\alpha(b,c)$, as desired.

The relations in $\Bbbk^\alpha M$ follow quickly: for example, we have
\begin{align*}
x^2 = \frac{\alpha(x,x)}{\alpha(y,z)} yz = (-q)^{B(x,\sigma^{-1}x) - B(y,\sigma^{-1}z)} yz = (-q)^{B(x,y) - B(y,x)} yz = (-q)^{1-0} yz = -q yz,
\end{align*}
and, similarly, $y^2 = -qzx$ and $z^2 = -qxy$, so that $\Bbbk^\alpha M = S_{q,0,1}$.

We now determine the Dehornoy quotient groups $G_n$. One can show that we have an identity
\begin{align*}
\Omega_{i+1}(x_1, \dots, x_{i+1}) = (x\;z\;y)^i(x_{i+1}),
\end{align*}
so that $r : X^2 \to X^2$ has minimal Dehornoy class $c=3$. Using this, we also find that
\begin{align*}
\nu(x^3) = \Omega_1(x) \Omega_2(x,x) \Omega_3(x,x,x) = xzy
\end{align*}
and hence $G_n$ is a group of order $(3n)^3 = 27n^3$ with presentation
\begin{align*}
G_n = \langle x,y,z \mid x^2=yz, \; y^2=zx, \; z^2=xy, \; (xzy)^n=1, \; (zyx)^n=1, \; (yxz)^n=1 \rangle.
\end{align*}
The last two relations are redundant, so $G_n \cong \Delta_n$.

By Theorem \ref{thm:DualReflection}, the subring of dual invariants with respect to $G_n$ is
\begin{align*}
\Bbbk \langle (xzy)^n, (yxz)^n, (zyx)^n \rangle
\end{align*}
and this is a quantum polynomial ring in three variables. If $\alpha$ descends to a $2$-cocycle on $G_n$, then the subring of dual invariants is commutative, again by Theorem \ref{thm:DualReflection}. Since 
\begin{align*}
1 = \alpha(1,y) = \alpha((xyz)^n, y) = (-q)^{B(x^{3n},y)} = (-q)^{3n},
\end{align*}
a necessary condition for $\alpha$ to descend to a $2$-cocycle on $G_n$ is that $(-q)^{3n} = 1$. One can show that this condition is also sufficient. 
\end{example}


\section{The centre of the structure algebra}
In this section, we prove a result which allows one to compute the centre of a structure algebra $A = \Bbbk M$ using commutative invariant theory. This result is useful in practice, as we shall demonstrate on the algebra from Example \ref{ex:BuzagloExample}.

\begin{theorem} \label{thm:Centre}
Write $H = \langle \psi_x \mid x \in X \rangle \leqslant \Sym(X)$ and $K = \{ a \in \NN^X \mid \psi_a = \id_X \}$. Also write $A = \Bbbk M$ and $R = \Bbbk[K] \subseteq \Bbbk [X]$. 
\begin{enumerate}[{\normalfont (1)},topsep=1pt,itemsep=1pt,leftmargin=35pt]
\item $K$ is a submonoid of $\NN^X$ and $R$ is a subring of $\Bbbk [X]$.
\item $K$, and hence $R$, are closed under the action of $H$.
\item Extend the map $\nu : \NN^X \to M$ $\Bbbk$-linearly to a map $\Bbbk[X] \to A$, which we also call $\nu$. Then $\nu$ restricts to an algebra homomorphism $R \to A$ and $\nu(R)$ is a commutative subalgebra of $A$. 
\item $\nu$ restricts to an isomorphism
\begin{align*}
R^H \xrightarrow{\cong} Z(A).
\end{align*}
\end{enumerate}
\end{theorem}
\begin{proof}
(1) Since $1 \in K$, it is nonempty. Now suppose that $a,b \in K$. Then, by Lemma \ref{lem:PsiToProduct}, we have
\begin{align*}
\psi_{ab} = \psi_{\psi_a(b)} \circ \psi_a = \psi_b = \id_X,
\end{align*}
so $ab \in \NN^X$. Thus $K$ is a submonoid of $M$, and the claim for $R$ also follows.

(2) Let $x \in X$ and $a \in K$. We want to show that $\psi_x(a) \in K$, i.e.~$\psi_{\psi_x(a)} = \id_X$. By Lemma \ref{lem:PsiToProduct}, we have
\begin{align*}
\psi_{\psi_x(a)} \circ \psi_x = \psi_{xa} = \psi_{ax} = \psi_{\psi_a(x)} \circ \psi_a = \psi_x,
\end{align*}
and hence $\psi_{\psi_x(a)} = \id_X$, as desired.

(3) It suffices to show that, if $a,b \in K$, then $\nu(ab) = \nu(a)\nu(b)$. Indeed, Lemma \ref{lem:PsiToProduct} gives
\begin{align*}
\nu(ab) = \nu(a) \nu(\psi_a(b)) = \nu(a) \nu(b).
\end{align*}
Therefore $\nu$ restricts to an algebra homomorphism $R \to A$ and, since $R$ is commutative, its image $\nu(R)$ is a commutative subalgebra of $A$.

(4) We first show that $\nu\big(R^H\big) \subseteq Z(A)$. Let 
\begin{align*}
p = \sum_{a \in K} \lambda_a \nu(a) \in R^H,
\end{align*}
where $\lambda_a \in \Bbbk$, only finitely many of which are nonzero. It suffices to show that $x\nu(p)=\nu(p)x$ in $A$ for all $x \in X$. Indeed, if $x \in X$ then
\begin{align*}
x \nu(p) &= \nu(x) \nu(p) = \nu(x \psi^{-1}_x(p)) = \nu(xp) = \nu(px)
= \nu\raisebox{-2pt}{$\Big($} \sum_{a \in K} \lambda_a ax \raisebox{-2pt}{$\Big)$}
= \sum_{a \in K} \lambda_a \nu(ax) \\
&= \sum_{a \in K} \lambda_a \nu(a) \psi_a(x)
= \sum_{a \in K} \lambda_a \nu(a) x
= \nu(p) x,
\end{align*}
where the third equality uses $H$-invariance of $p$, and the penultimate equality uses the fact that $\psi_a = \id_X$ for all $a \in K$. Therefore $\nu(p) \in Z(A)$, so $R^H \subseteq Z(A)$.

Seeking to show the reverse inclusion, let $z \in Z(A)$ be arbitrary, and write
\begin{align*}
z = \sum_{a \in \NN^X} \lambda_a \nu(a) \quad \text{and} \quad p = \sum_{a \in \NN^X} \lambda_a a
\end{align*}
for some $\lambda_a \in \Bbbk$, only finitely many of which are nonzero. We first show that the support of $z$ is contained in $K$. To this end, for each $\sigma \in H$, let
\begin{align*}
z_\sigma = \sum_{\psi_a = \sigma} \lambda_a \nu(a) \quad \text{and} \quad p_\sigma = \sum_{\psi_a = \sigma} \lambda_a a
\end{align*}
so that $z = \sum_{\sigma \in H} z_\sigma$. If $x \in X$, write $a_x = \nu(x^c)$, where $c$ is the minimal Dehornoy class of $X$. By \cite[Lemma 5.5]{dehornoy}, we have $\psi_{x^c} = \id_X$, and Lemma \ref{lem:PsiToProduct} shows that $\psi_{ax^c} = \psi_a$ for all $a \in \NN^X$. Then, if $a \in \NN^X$ is arbitrary, we have
\begin{gather*}
a_x \nu(a) = \nu(x^c) \nu(a) = \nu(x^c \psi_{x^c}^{-1}(a)) = \nu(x^c a) \\
\nu(a) a_x = \nu(a) \nu(x^c) = \nu(a \psi_{a}^{-1}(x^c)) = \nu(a \psi_a^{-1}(x)^c).
\end{gather*}
Writing $y \coloneqq \psi_a^{-1}(x) \in X$, this shows that $H$-components in which $a_x \nu(a)$ and $\nu(a) a_x$ lie are $\psi_{x^c a}$ and $\psi_{ay^c}$, both of which are equal to $\psi_a$. Therefore $a_x z = z a_x$ if and only if $a_x z_\sigma = z_\sigma a_x$ for all $\sigma \in H$.

Now fix $\sigma \in H$. If $x \in X$, then we have
\begin{align*}
\nu(p_\sigma x^c)
&= \nu\raisebox{-2pt}{$\Big($} \sum_{\psi_a=\sigma} \lambda_a ax^c \raisebox{-2pt}{$\Big)$}
= \sum_{\psi_a=\sigma} \lambda_a \nu(x^c a)
= \nu(x^c) \sum_{\psi_a=\sigma} \lambda_a \nu(a) 
= a_x z_\sigma
= z_\sigma a_x \\
&= \sum_{\psi_a=\sigma} \lambda_a \nu(a) \nu(x^c) 
= \sum_{\psi_a=\sigma} \lambda_a \nu(a \psi_a^{-1}(x^c))
= \nu\raisebox{-2pt}{$\Big($} \sum_{\psi_a=\sigma} \lambda_a a \sigma^{-1}(x)^c \raisebox{-2pt}{$\Big)$}
= \nu(p_\sigma \sigma^{-1}(x)^c)
\end{align*}
Since $\nu$ is injective, we get $p_\sigma (x^c - \sigma^{-1}(x)^c) = 0$ in the domain $\Bbbk[X]$. If $\sigma \neq \id$ then, choosing $x \in X$ with $\sigma(x) \neq x$, we have $(x^c - \sigma^{-1}(x)^c) \neq 0$, and hence $p_\sigma = 0$. Thus $p = p_{\id_X}$ and $z = z_{\id_X}$ have support contained in $K$.

Finally, we show that $p$ is invariant under $H$, so that $p \in R^H$. If $x \in X$ then
\begin{align*}
\nu(xp) 
&= \nu(px) 
= \sum_{\psi_a=\id_X} \lambda_a \nu(ax)
= \sum_{\psi_a=\id_X} \lambda_a \nu(a) \psi_a(x)
= \sum_{\psi_a=\id_X} \lambda_a \nu(a) x \\
&= zx
= xz
= \sum_{\psi_a=\id_X} \lambda_a x \nu(a)
= \sum_{\psi_a=\id_X} \lambda_a \nu(x \psi_x^{-1}(a))
= \nu(x \psi_x^{-1}(p)).
\end{align*}
Since $\nu$ is injective, we get $xp = x \psi_x^{-1}(p)$, and hence $p = \psi_x^{-1}(p)$. This holds for all $x \in X$, so $p \in R^H$ and hence $Z(A) \subseteq \nu(R^H)$.
\end{proof}

The above theorem makes it practical to compute the centre of $\Bbbk M$ in many examples. The group $H$ is easily computed and, writing $X = \{x_1, \dots, x_n\}$, a generating set for $K$ is contained in $\{ x_1^{i_1} \dots x_n^{i_n} \mid 0 \leqslant i_j < c \}$. From here, the invariant ring $\Bbbk[K]^H$ can be determined using standard techniques from commutative invariant theory.

\begin{example} \label{ex:BuzagloCentre}
We return to Example \ref{ex:BuzagloExample} and calculate the centre of \begin{align*}
\Bbbk M =
\frac{\Bbbk \langle x,y,z \rangle}{\langle \, x^2-yz, \; y^2-zx, \; z^2-xy \,\rangle}.
\end{align*}
Since $\psi_a = (x\;z\;y)$ for all $a \in X$, we have $H = \langle (x\;z\;y) \rangle \leqslant \Sym(X)$. To determine $K$, we note that, if $a,b \in X$ then
\begin{align*}
\psi_{ab} = \psi_{\psi_a(b)} \circ \psi_a = (x\;z\;y)^2,
\end{align*}
and, inductively, if $a \in \NN^X$ has length $k$ then $\psi_a = (x\;z\;y)^k$. Accordingly, 
\begin{align*}
K = \{a \in \NN^X \mid \ell(a) \equiv 0 \text{ mod } 3 \}.
\end{align*}
Therefore $\Bbbk[K] = \Bbbk[x,y,z]^{(3)}$, the $3$-Veronese of a polynomial ring in three variables.

We now determine $\Bbbk[K]^H = (\Bbbk[x,y,z]^{(3)})^H$. We start by determining the Hilbert series of $\Bbbk[K]$. Since $\Bbbk[K] = \Bbbk[x,y,z]^{\langle g \rangle}$, where $g = \begin{psmallmatrix} \omega & 0 & 0 \\ 0 & \omega & 0 \\ 0 & 0 & \omega \end{psmallmatrix}$ and $\omega$ is a primitive third root of unity, Molien's theorem gives
\begin{align*}
\hilb_{\Bbbk[K]}(t) = \frac{1}{3}\bigg(\frac{1}{(1-t)^3} + \frac{1}{(1-\omega t)^3} + \frac{1}{(1-\omega^2t)^3}\bigg) = \frac{1-7t^3+t^6}{(1-t^3)^3}.
\end{align*}
We can determine the Hilbert series of $\Bbbk[K]^H$ in a similar way. In each degree $d=3k$ for which $\Bbbk[K]$ is nonzero, the generator $\sigma \in H$ permutes the basis elements, and there is a unique fixed point, namely $x^k y^k z^k$. Consequently, Molien's Theorem, in its formulation involving the trace, gives
\begin{align*}
\hilb_{\Bbbk[K]^H}(t) = \frac{1}{3}\bigg(\frac{1-7t^3+t^6}{(1-t^3)^3} + \frac{1}{(1-t^3)} + \frac{1}{(1-t^3)}\bigg) = \frac{1-t^9}{(1-t^3)^4}.
\end{align*}

We now determine a generating set for $\Bbbk[K]^H$. Write
\begin{align*}
a = x+y+z, \quad b = x+\omega y + \omega^2 z, \quad c = x+\omega^2y+\omega z,
\end{align*}
which satisfy
\begin{align*}
\sigma(a) = a, \quad \sigma(b) = \omega b, \quad \sigma(c) = \omega^2 c.
\end{align*}
Consequently, we have $a^3, b^3, c^3, abc \in \Bbbk[K]^H$. Define $S = \Bbbk[a^3,b^3,c^3,abc] \subseteq \Bbbk[K]^H$ and $T = \Bbbk[a^3,b^3,c^3]$, the latter of which is a weighted polynomial ring. By degree considerations, we have
\begin{align*}
S = T \oplus T(abc) \oplus T(abc)^2, 
\end{align*}
and hence
\begin{align*}
\hilb_S(t) = (1+t^3+t^6) \hilb_T(t) = \frac{1-t^9}{(1-t^3)^4} = \hilb_{\Bbbk[K]^H}(t).
\end{align*}
Since $S \subseteq \Bbbk[K]^H$ and they have the same Hilbert series, we have equality. In particular, it is clear that the generators of $\Bbbk[K]^H$ satisfy a single relation in degree $9$, so that
\begin{align*}
Z(\Bbbk M) \cong \frac{\Bbbk[A,B,C,D]}{\langle ABC - D^3\rangle}.
\end{align*}
We may choose explicit elements of $\Bbbk M$ which satisfy this relation by applying $\nu$ to the corresponding elements of $\Bbbk[K]^H$:
\begin{align*}
A &= \nu(a^3) = (xzy + yxz + zyx) + 3(xz^2 + yx^2 + zy^2) + 3(x^2z + y^2x + z^2y) + 6x^3, \\
B &= \nu(a^3) = (xzy + yxz + zyx) + 3\omega(xz^2 + yx^2 + zy^2) + 3\omega^2(x^2z + y^2x + z^2y) + 6x^3, \\
C &= \nu(a^3) = (xzy + yxz + zyx) + 3\omega^2(xz^2 + yx^2 + zy^2) + 3\omega(x^2z + y^2x + z^2y) + 6x^3, \\
D &= \nu(a^3) = (xzy + yxz + zyx) - 3x^3. 
\end{align*}
One can verify that the elements $A,B,C,D \in \Bbbk M$ indeed satisfy $ABC=D^3$.
\end{example}

Theorem \ref{thm:DualReflection} shows that $(\Bbbk M)_{1_{G_k}}$ forms a large commutative subring of $\Bbbk M$. It is natural to ask when this subring coincides with the centre of $\Bbbk M$.

\begin{lemma} \label{lem:Centre}
Suppose $r : X \to X$ is a solution to the Yang--Baxter equation and $k \geqslant 1$. Then $(\Bbbk M)_{1_{G_k}} = Z(\Bbbk M)$ if and only if $r(x,y) = (y,x)$ for all $x,y \in X$ and $k=1$.
\end{lemma}
\begin{proof}
For sufficiency, if $r(x,y) = (y,x)$ for all $x,y \in X$, then $\Bbbk M$ is a polynomial ring and $G$ is trivial, so $(\Bbbk M)_{1_{G_k}} = Z(\Bbbk M)$.

Conversely, assume $(\Bbbk M)_{1_{G_k}} = Z(\Bbbk M)$ so that, if $x,y \in X$, then $\nu(x^c)^k y = y \nu(x^c)^k$. Therefore
\begin{align*}
\nu(x^{ck} y) &= \nu(x^{ck}) \psi_{x^{ck}}(y) = \nu(x^{ck}) y = y \nu(x^c)^k = \nu(y) \nu(x^c)^k = \nu(y) \nu(x^{ck}) = \nu(y \, \psi^{-1}_y(x^{ck})) \\
&= \nu( y\, \psi^{-1}_y(x)^{ck} ),
\end{align*}
where the second equality uses \cite[Lemma 5.5]{dehornoy} and the penultimate equality uses Lemma \ref{lem:PsiToProduct}. Since $\nu$ is a bijection, we have $x^{ck} y = \psi_y(x)^{ck} y$ in $\NN^X$, and cancellation gives $\psi^{-1}_y(x) = x$. This holds for all $x,y \in X$, and \cite[Proposition 1.5, Lemma 4.4]{dehornoy} together imply that $r(x,y) = (y,x)$ for all $x,y \in X$. In particular, $M$ is the free abelian monoid generated by $X$, $G_k \cong C_k^n$, where $n = |X|$, and
\begin{align*}
\Bbbk[x_1, \dots, x_n] = Z(\Bbbk M) = (\Bbbk M)_{1_{G_k}} = \Bbbk[x_1^k, \dots, x_n^k].
\end{align*}
Therefore $k=1$, as desired.
\end{proof}


\section{The Koszul dual and the Koszul complex}

Throughout this section, write $A = \Bbbk^\alpha M$ for the twisted structure algebra associated to some solution $r : X^2 \to X^2$ to the Yang--Baxter equation. The goal of this section is to give an explicit description of the Koszul dual of $A^!$ of $A$, and to write down explicitly the Koszul complex for $A$.

Write $V$ for the $\Bbbk$-vector space with basis $X$; we shall write $X^\vee = \{ x^\vee \mid x \in X \}$ for the corresponding dual basis of $V^*$.

\begin{lemma} \label{lem:KoszulDual}
A presentation for the Koszul dual of $A$ is
\begin{align*}
A^! = \frac{\Bbbk\langle\, X^\vee \,\rangle}{
\bigg\langle\begin{array}{cl}
y^\vee x^\vee + \frac{\alpha(u,v)}{\alpha(x,y)} v^\vee u^\vee & \text{ if } r(x,y) = (u,v) \neq (x,y) \\
y^\vee x^\vee = 0 & \text{ if } r(x,y)=(x,y)
\end{array}
\bigg\rangle}.
\end{align*}
\end{lemma}
\begin{proof}
Let $\mathcal{R} \subseteq V \otimes V$ be the subspace spanned by the $\binom{n}{2}$ linearly independent relations in the defining ideal of $A$, so $\dim_\Bbbk \mathcal{R} = \binom{n}{2}$. We wish to determine $\mathcal{R}^\perp \subseteq V^* \otimes V^*$. We recall our convention that that the pairing 
$(V^* \otimes V^*) \otimes (V \otimes V) \to \Bbbk$ is given by $(f \otimes g)(u \otimes v) \mapsto f(v) g(u)$.

If $r(x,y) = (x,y)$, then the monomial $xy$ does not appear in any of the defining relations in $\mathcal{R}$, so $y^\vee x^\vee$ vanishes on $\mathcal{R}$. On the other hand, if $r(x,y) = (u,v) \neq (x,y)$ then $y^\vee x^\vee + \frac{\alpha(u,v)}{\alpha(x,y)} v^\vee u^\vee$ vanishes on the relation $xy - \frac{\alpha(x,y)}{\alpha(u,v)} uv$ in $A$, essentially by construction, and it vanishes on every other relation because none of them involve the monomials $xy$ or $uv$. 

We have therefore identified $\binom{n}{2} + n = \binom{n+1}{2}$ linearly independent elements of $\mathcal{R}^\perp$. Since $\mathcal{R}^\perp$ is the annihilator of $\mathcal{R}$ inside $V \otimes V$, it has dimension $\dim_\Bbbk \mathcal{R}^\perp = n^2 - \binom{n}{2} = \binom{n+1}{2}$, so we are done.
\end{proof}

We note that the relations in $A^!$ could just be given as
\begin{align*}
y^\vee x^\vee + \frac{\alpha(u,v)}{\alpha(x,y)} v^\vee u^\vee = 0 & \text{ if } r(x,y) = (u,v)
\end{align*}
since, if $r(x,y) = (x,y)$, then the corresponding relation becomes $2 y^\vee x^\vee = 0$, and therefore $y^\vee x^\vee = 0$ since we work over a field.

The Koszul dual $A^!$ is $\NN$-graded with $\deg x^\vee = 1$ for all $x \in X$. It also has a natural $G_k$-grading, for any $k \geqslant 1$,  obtained by declaring $\deg x = \pi_k(x)^{-1} \in G_k$ for all $x \in X$. The negative exponent ensures that the natural evaluation map
\begin{align*}
(A^!)_i \otimes (A)_i \to \Bbbk, \quad (x_i^\vee \otimes \dots \otimes x_1^\vee) \otimes (x_1 \otimes \dots \otimes x_i) \mapsto x^\vee_1(x_1) \dots x^\vee_i(x_i)
\end{align*}
is a morphism of $(\Bbbk G_k)^*$-modules. 

We define a map $\gamma$ on tuples of elements of $X$ as follows:
\begin{align*}
\gamma_\alpha : X^* \to \Bbbk, \quad \gamma_\alpha(x_1, \dots, x_m) = \prod_{j=1}^{m-1} \alpha( x_1 \dots x_j, x_{j+1}).
\end{align*}
This has the following useful property:

\begin{lemma} \label{lem:GammaProperty}
Suppose that $\mathbf{x} = (x_1, \dots, x_m)$ and $\widetilde{\mathbf{x}} = (\widetilde{x}_1, \dots, \widetilde{x}_m)$ differ only in coordinates $i$ and $i+1$, and in these positions, we have $x_i x_{i+1} = \widetilde{x}_i \widetilde{x}_{i+1}$ in $M$. Then
\begin{align*}
\frac{\gamma_\alpha(\mathbf{x})}{\gamma_\alpha(\widetilde{\mathbf{x}})} = \frac{\alpha(x_i,x_{i+1})}{\alpha(\widetilde{x}_i,\widetilde{x}_{i+1})}.
\end{align*}
\end{lemma}

\begin{proof}
Write
\begin{align*}
u = \alpha(x_1,x_2) \dots \alpha(x_1 \dots x_{i-2}, x_{i-1}) \quad \text{and} \quad 
v = \alpha(x_1 \dots x_{i+1}, x_{i+2}) \dots \alpha(x_1 \dots x_{n-1}, x_n).
\end{align*}
Since $x_i x_{i+1} = \widetilde{x}_i \widetilde{x}_{i+1}$, we have
\begin{gather*}
\gamma_\alpha(\mathbf{x}) = u \; \alpha(x_1 \dots x_{i-1}, x_i) \alpha(x_1 \dots x_{i-1} x_i, x_{i+1}) \; v, \\
\gamma_\alpha(\widetilde{\mathbf{x}}) = u \; \alpha(x_1 \dots x_{i-1}, \widetilde{x}_i) \alpha(x_1 \dots x_{i-1} \widetilde{x}_{i}, \widetilde{x}_{i+1}) \; v.
\end{gather*}
Thus, writing $p = x_1 \dots x_{i-1}$, the cocycle identity gives
\begin{align*}
\frac{\gamma_\alpha(\mathbf{x})}{\gamma_\alpha(\widetilde{\mathbf{x}})}
= \frac{\alpha(p,x_i)\, \alpha(p x_i, x_{i+1})}{\alpha(p,\widetilde{x}_i)\, \alpha(p \widetilde{x}_{i}, \widetilde{x}_{i+1})}
= \frac{\alpha(p,x_i x_{i+1})\, \alpha(x_i, x_{i+1})}{\alpha(p,\widetilde{x}_i \widetilde{x}_{i+1})\, \alpha(\widetilde{x}_{i}, \widetilde{x}_{i+1})}
= \frac{\alpha(x_i,x_{i+1})}{\alpha(\widetilde{x}_i,\widetilde{x}_{i+1})},
\end{align*}
as desired.
\end{proof}

We use this to establish a bijection between the exterior algebra of $V$ and $A^!$. We caution that, in general, this map is not an isomorphism of $\Bbbk$-algebras.

\begin{lemma}
The map
\begin{align*}
\eta : \exterior V \to A^!, \quad \eta(x_1 \wedge \dots \wedge x_m) &= \gamma_\alpha(y_1, \dots, y_m) \,y_m^\vee \dots y_1^\vee, \\
&\hspace{30pt}\text{ where } \Pi_m(x_1, \dots, x_m) = (y_1, \dots, y_m),
\end{align*}
is a bijection of $\NN$-graded vector spaces.
\end{lemma}
\begin{proof}
Seeking to show that $\eta$ is well-defined, we define 
\begin{align*}
\overline{\eta} : \exterior V \to A^!, \quad \overline{\eta}(x_1 \otimes \dots \otimes x_m) &= \gamma_\alpha(y_1, \dots, y_m) \,y_m^\vee \dots y_1^\vee, \\
&\hspace{30pt}\text{ where } \Pi_m(x_1, \dots, x_m) = (y_1, \dots, y_m),
\end{align*}
This map is surjective: if $y_m^\vee \dots y_1^\vee \in A^!$ then, since $\Pi_m$ is bijective, there exist $x_1, \dots, x_m \in X$ with $\Pi_m(x_1, \dots, x_n) = (y_1, \dots, y_m)$, so that $\overline{\eta}(\gamma_\alpha(y_1, \dots, y_m)^{-1} x_1 \otimes \dots \otimes x_m) = y_m^\vee \dots y_1^\vee$. 

We now show that $\overline{\eta}$ is alternating. Consider $\mathbf{x} = (x_1, \dots, x_m)$ and $\widetilde{\mathbf{x}} = (x_{\sigma(1)}, \dots, x_{\sigma(m)})$ where $\sigma = (i\;i+1)$ for some $1 \leqslant i < m$. Let $\mathbf{y} = \Pi_m(\mathbf{x})$ and $\widetilde{\mathbf{y}} = \Pi_m(\widetilde{\mathbf{x}})$. By the proof of \cite[Lemma 2.6]{dehornoy}, $\mathbf{y}$ and $\widetilde{\mathbf{y}}$ differ only in coordinates $i$ and $i+1$, and these coordinates satisfy $y_i y_{i+1} = \widetilde{y}_i \widetilde{y}_{i+1}$ in $M$. By Lemma \ref{lem:GammaProperty}, we have $\frac{\gamma_\alpha(\mathbf{y})}{\gamma_\alpha(\widetilde{\mathbf{y}})} = \frac{\alpha(y_i,y_{i+1})}{\alpha(\widetilde{y}_i,\widetilde{y}_{i+1})}$. Therefore,
\begin{align*}
\overline{\eta}(x_{\sigma(1)} \otimes \dots \otimes x_{\sigma(m)})
&= \gamma_\alpha(\widetilde{\mathbf{y}}) \; \widetilde{y}_m^\vee \dots \widetilde{y}_1^\vee \\
&= \gamma_\alpha(\widetilde{\mathbf{y}}) \; y_m^\vee \dots y_{i+2}^\vee  \; \widetilde{y}_{i+1}^\vee \widetilde{y}_{i}^\vee \; y_{i-1}^\vee \dots y_1^\vee \\
&= -\gamma_\alpha(\widetilde{\mathbf{y}}) \frac{\alpha(y_i,y_{i+1})}{\alpha(\widetilde{y}_i,\widetilde{y}_{i+1})}\; y_m^\vee \dots y_{i+2}^\vee  \; y_{i+1}^\vee y_{i}^\vee \; y_{i-1}^\vee \dots y_1^\vee \\
&= -\gamma_\alpha(\mathbf{y}) \; y_m^\vee \dots y_1^\vee \\
&= -\overline{\eta}(x_{1} \otimes \dots \otimes x_{m}).
\end{align*}
Thus $\overline{\eta}$ is alternating, and hence descends to a well-defined surjective map $\eta$ as in the statement of the lemma. Finally, since both $\exterior V$ and $A^!$ have Hilbert series $(1+t)^n$, $\eta$ is an isomorphism of $\NN$-graded vector spaces.
\end{proof}

Pick an ordering on the elements of $X$, say $x_1 \prec x_2 \prec \dots \prec x_n$. Using $\eta$, we can choose a distinguished basis of $A^!$: if $I \subseteq X$, with $I = \{x_{i_1}, \dots, x_{i_m}\}$ and $x_{i_1} \prec x_{i_2} \prec \dots \prec x_{i_m}$, then we write $y_I = \eta(x_{i_1} \wedge \dots \wedge x_{i_m})$. By the previous lemma, $\{ y_I \mid I \subseteq X \}$ is a basis of $A^!$. We will also write $\phi_I \in (A_{|I|}^!)^*$ for the element which is dual to $y_I$.

\begin{proposition} \label{prop:KoszulComplex}
The differentials in the Koszul complex for $A$ are given by
\begin{gather*}
\partial_m : (A_m^!)^* \otimes A \to (A_{m-1}^!)^* \otimes A, \\ \partial_m(\phi_I \otimes a) = \sum_{j=1}^m (-1)^{m-j} \alpha\big(\nu(x_{i_1} \dots \widehat{x}_{i_j} \dots x_{i_m}), x_{I,j} \big)^{-1} \phi_{I\setminus\{x_{i_j}\}} \otimes x_{I,j}\, a,
\end{gather*}
where $x_{I,j} \coloneqq \Omega_m(x_{i_1}, \dots, \widehat{x}_{i_j}, \dots, x_{i_m}, x_{i_j})$ and $\widehat{x}_{i_j}$ indicates that we omit the element $x_{i_j}$. In particular, if $\alpha$ is trivial, then
\begin{align*}
\partial_m(\phi_I \otimes a) = \sum_{j=1}^m (-1)^{m-j} \phi_{I\setminus\{x_{i_j}\}} \otimes x_{I,j}\, a.
\end{align*}
\end{proposition}
\begin{proof}
By definition, we have
\begin{align*}
\partial_m(\phi_I \otimes a) = \sum_{x \in X} \phi_I(x^\vee -) \otimes xa,
\end{align*}
so we need to determine $\phi_I(x^\vee y_J)$ as $J$ varies over subsets of size $m-1$. 

Fix $I$ as in the statement. We first establish some notation. Write
\begin{align*}
\Pi_m(x_{i_1}, \dots, x_{i_m}) = (y_1, \dots, y_m) \quad \text{and} \quad \Pi_{m-1}(x_{i_1}, \dots, \widehat{x}_{i_j}, \dots, x_{i_m}) = (z_1, \dots, z_{m-1}).
\end{align*}
With $x_{I,j}$ as in the statement of the proposition, this means that $\Pi_m(x_{i_1}, \dots, \widehat{x}_{i_j}, \dots, x_{i_m},x_{i_j}) = (z_1, \dots, z_{m-1}, x_{I,j})$. Also write $z_{I,j} = z_1 \dots z_{m-1} = \nu(x_{i_1}, \dots, \widehat{x}_{i_j}, \dots, x_{i_m}) \in M$. We then have
\begin{align*}
\eta(x_{i_1} \wedge \dots \wedge \widehat{x}_{i_1} \wedge \dots \wedge x_{i_m} \wedge x_{i_j}) = (-1)^{m-j} \eta (x_{i_1} \wedge \dots \wedge x_{i_m}) = (-1)^{m-j} y_I.
\end{align*}
We can also calculate this as
\begin{align*}
\eta(x_{i_1} \wedge \dots \wedge \widehat{x}_{i_1} \wedge \dots \wedge x_{i_m} \wedge x_{i_j}) &= \gamma_\alpha(z_1, \dots, z_{m-1},x_{I,j}) \,x_{I,j}^\vee z_{m-1}^\vee \dots z_1^\vee \\
&= \gamma_\alpha(z_1, \dots, z_{m-1}) \alpha(z_{I,j}, x_{I,j}) \, x_{I,j}^\vee z_{m-1}^\vee \dots z_1^\vee \\
&= \alpha(z_{I,j}, x_{I,j})\, x_{I,j}^\vee \, y_{I\setminus\{x_{i_j}\}}
\end{align*}
Comparing these expressions gives
\begin{align*}
x_{I,j}^\vee \, y_{I\setminus\{x_{i_j}\}} = (-1)^{m-j} \alpha(z_{I,j}, x_{I,j})^{-1} y_I.
\end{align*}

Using this, we can determine $\phi_I(x^\vee y_J)$. Indeed, the above shows that this quantity is nonzero if and only if $J = I \setminus \{x_{i_j}\}$ and $x = x_{I,j}$ for some $1 \leqslant j \leqslant m$. In this case, we have
\begin{align*}
\phi_I(x^\vee y_J) 
= \phi_I(x_{I,j}^\vee y_{I \setminus \{x_{i_j}\}}) 
= \phi_I\big( (-1)^{m-j} \alpha(z_{I,j}, x_{I,j})^{-1} y_I \big)
= (-1)^{m-j} \alpha(z_{I,j}, x_{I,j})^{-1}
\end{align*}
so that $\phi_I(x^\vee -) = (-1)^{m-j} \alpha(z_{I,j}, x_{I,j})^{-1} \phi_{I \setminus \{x_{i_j}\}}$. Therefore
\begin{align*}
\partial_m(\phi_I \otimes a) 
= \sum_{x \in X} \phi_I(x^\vee -) \otimes xa
= \sum_{j=1}^m (-1)^{m-j} \alpha(z_{I,j}, x_{I,j})^{-1} \phi_{I \setminus \{x_{i_j}\}} \otimes x_{I,j}\, a,
\end{align*}
which is the formula claimed in the statement of the proposition.
\end{proof}

As a sanity check, let us compute these maps for $m=1$ and $m=2$ and verify that their composition is the zero map. When $m=1$, we have
\begin{align*}
\partial_1(\phi_{x_i} \otimes a) = x_i a, 
\end{align*}
after making the identification $\Bbbk \otimes A \cong A$. Moreover, if $I = \{x_i \prec x_j \}$ then
\begin{align*}
\partial_2 (\phi_{\{x_i,x_j\}} \otimes a) = -\alpha(x_j, x_j * x_i)^{-1} \phi_{x_j} \otimes (x_j * x_i) a + \alpha(x_i, x_i * x_j)^{-1} \phi_{x_i} \otimes (x_i * x_j) a.
\end{align*}
Composing these maps gives
\begin{align*}
\partial_1 \partial_2 (\phi_{\{x_i,x_j\}} \otimes a)
= \big(-\alpha(x_j, x_j * x_i)^{-1} x_j (x_j * x_i) + \alpha(x_i, x_i * x_j)^{-1} x_i (x_i * x_j) \big) a,
\end{align*}
which vanishes in $\Bbbk^\alpha M$.

\begin{example}
Let us compute the Koszul dual and Koszul complex for the algebra from Example \ref{ex:BuzagloExample}. By Lemma \ref{lem:KoszulDual}, the Koszul dual of $A$ has presentation
\begin{align*}
A^! = \frac{\Bbbk\langle x^\vee, y^\vee, z^\vee \rangle}{
\left\langle
\begin{array}{ccc}
x^\vee x^\vee + q^{-1} z^\vee y^\vee, & y^\vee y^\vee + q^{-1} x^\vee z^\vee, & z^\vee z^\vee + q^{-1} y^\vee x^\vee \\
x^\vee y^\vee, & y^\vee z^\vee, & z^\vee x^\vee 
\end{array}
\right\rangle
}.
\end{align*}

When calculating the maps in the Koszul complex, we shall write, for example, $\phi_{xy}$ instead of $\phi_{\{x,y\}}$ for the element of $(A_2^!)^*$ which is dual to $\eta(x \wedge y)$. In this notation, the discussion after Proposition \ref{prop:KoszulComplex} shows that
\begin{align*}
\partial_1(\phi_x \otimes a) = xa, \quad \partial_1(\phi_y \otimes a) = ya, \quad \partial_1(\phi_z \otimes a) = za
\end{align*}
and, for example,
\begin{align*}
\partial_2(\phi_{xy} \otimes a) 
&= - \alpha(\nu(y), y*x)^{-1} \phi_y \otimes (y*x) a + \alpha(\nu(x), x*y)^{-1} \phi_x \otimes (x*y) a \\
&= -\alpha(y,z)^{-1} \phi_y \otimes za + \alpha(x,x)^{-1} \phi_x \otimes xa \\
&= - \phi_y \otimes za - q^{-1} \phi_x \otimes xa.
\end{align*}
Similar calculations give
\begin{align*}
\partial_2(\phi_{xz}) = q^{-1} \phi_z \otimes za + \phi_x \otimes ya, \quad \partial_2(\phi_{yz}) = -\phi_z \otimes xa - q^{-1} \phi_y \otimes ya. 
\end{align*}
Finally, 
\begin{align*}
\partial_3(\phi_{xyz} \otimes a) 
&= \alpha(\nu(yz), \Omega_3(y,z,x))^{-1} \phi_{yz} \otimes \Omega_3(y,z,x) a \\
&\hspace{30pt}- \alpha(\nu(xz), \Omega_3(x,z,y))^{-1} \phi_{xz} \otimes \Omega_3(x,z,y) a \\
&\hspace{30pt}+ \alpha(\nu(xy), \Omega_3(x,y,z))^{-1} \phi_{xy} \otimes \Omega_3(x,y,z) a \\
&= \alpha(y^2,y)^{-1} \phi_{yz} \otimes y a - \alpha(xy,z)^{-1} \phi_{xz} \otimes z a - \alpha(x^2,x)^{-1} \phi_{xy} \otimes x a \\
&= -q^{-1} \phi_{yz} \otimes y a + q^{-1} \phi_{xz} \otimes z a - q^{-1} \phi_{xy} \otimes x a.
\end{align*}
In terms of matrices, these maps are given by
\begin{align*}
\partial_1 = 
\begin{pmatrix}
x & y & z
\end{pmatrix},
\qquad
\partial_2 = 
\begin{pmatrix}
-q^{-1} x & y & 0 \\
-z & 0 & -q^{-1} y \\
0 & q^{-1} z & -x
\end{pmatrix},
\qquad
\partial_3 = 
\begin{pmatrix}
-q^{-1} x \\ q^{-1}z \\ -q^{-1} y
\end{pmatrix},
\end{align*}
from which it is quick to verify that adjacent maps compose to give the zero map.
\end{example}


\section{Superpotentials for Structure Algebras}

Let $r$ be a solution to the Yang--Baxter equation, $M = M(X,r)$ its structure monoid, and $A = \Bbbk^\alpha M$ its (twisted) structure algebra. Since $A$ is AS regular and Koszul, is is necessarily a superpotential algebra by \cite[Theorem 11]{duboisviolette}. We now show that there is a simple formula for its defining superpotential.

It will be convenient to introduce some notation. Throughout, assume that we have fixed an ordering for $X$, and write $X = \{x_1, \dots, x_n\}$. If $\sigma \in S_n$, write
\begin{align}
y_i^\sigma &\coloneqq \Omega_i(x_{\sigma(1)}, \dots, x_{\sigma(i)}) \nonumber \\
&= \psi_{x_{\sigma(1)} \dots x_{\sigma(i-1)}}(x_{\sigma(i)}), \label{eqn:AltForyi}
\end{align}
where \eqref{eqn:AltForyi} follows from \cite[Lemma 4.5]{dehornoy}. We note that, by Lemma \ref{lem:PropertiesOfDelta}, we have
\begin{align*}
\Delta = y_1^\sigma y_2^\sigma \dots y_n^\sigma.
\end{align*}
We also write
\begin{align*}
\gamma_\alpha^\sigma = \alpha(y_1^\sigma, y_2^\sigma) \alpha(y_1^\sigma y_2^\sigma, y_3^\sigma) \dots \alpha(y_1^\sigma \dots y_{n-1}^\sigma, y_n^\sigma).
\end{align*}

We begin by establishing an important property of $y_i^\sigma$.

\begin{lemma}
In $M$, we have $y_i^\sigma y_{i+1}^\sigma = y_{i+1}^\sigma y_i^\sigma$ for $1 \leqslant i < n$. 
\end{lemma}
\begin{proof}
Using the definition of $\Omega_i$, we obtain
\begin{align*}
y_i^\sigma y_{i+1}^\sigma 
&= \Omega_i(x_{\sigma(1)}, \dots, x_{\sigma(i)}) \Omega_{i+1}(x_{\sigma(1)}, \dots, x_{\sigma(i)}, x_{\sigma(i+1)}) \\
&= \Omega_i(x_{\sigma(1)}, \dots, x_{\sigma(i)}) \big( \Omega_{i}(x_{\sigma(1)}, \dots, x_{\sigma(i)}) * \Omega_{i}(x_{\sigma(1)}, \dots, x_{\sigma(i-1)}, x_{\sigma(i+1)}) \\
&= \Omega_{i}(x_{\sigma(1)}, \dots, x_{\sigma(i-1)}, x_{\sigma(i+1)})  \big( \Omega_{i}(x_{\sigma(1)}, \dots, x_{\sigma(i-1)},x_{\sigma(i+1)}) * \Omega_i(x_{\sigma(1)}, \dots, x_{\sigma(i)}) \big) \\
&= \Omega_{i}(x_{\sigma(1)}, \dots, x_{\sigma(i-1), x_{\sigma(i+1)}} \Omega_{i+1}(x_{\sigma(1)}, \dots, x_{\sigma(i-1)}, x_{\sigma(i+1)}, x_{\sigma(i)}\big) \\
&= y_{i+1}^\sigma y_i^\sigma,
\end{align*}
where the equality on the third line uses the fact that the relations in $M$ may be presented as in \eqref{eqn:StructureMonoidAlt}.
\end{proof}

We are now in a position to determine a twisted superpotential $\sfw_\alpha$ which defines $A$. The following result shows that $\sfw_\alpha$ can be obtained as a suitable linear combination of all of the factorisations of the Garside element $\Delta$.

\begin{proposition} \label{prop:Superpotential}
If $\sigma \in S_n$, write $\sfw_\sigma = y_1^\sigma \otimes \dots \otimes y_n^\sigma$. A twisted superpotential $\sfw_\alpha$ which defines $A$, in the sense that $A \cong \mathscr{D}(\sfw_\alpha,n-2)$, is given by
\begin{align*}
\sfw_\alpha = \sum_{\sigma \in S_n} \frac{\operatorname{sgn} \sigma}{\gamma_\alpha^\sigma} \sfw_\sigma = \sum_{\sigma \in S_n} \frac{\operatorname{sgn} \sigma}{\gamma_\alpha^\sigma} y_1^\sigma \otimes \dots \otimes y_n^\sigma.
\end{align*}
\end{proposition}
\begin{proof}
Write $V = \spann_\Bbbk X$ and 
\begin{align*}
\mathcal{R} = \spann_\Bbbk \bigg\{ x \otimes y - \frac{\alpha(x,y)}{\alpha(u,v)}u \otimes v \;\bigg|\; r(x,y) = (u,v) \text{ where } x,y,u,v \in X \bigg\} \subseteq V^{\otimes 2},
\end{align*}
so that $A = T(V)/\langle \mathcal{R} \rangle$. For $1 \leqslant i \leqslant n-1$, write
\begin{align*}
W_i \coloneqq V^{\otimes i-1} \otimes \mathcal{R} \otimes V^{\otimes (n-1-i)} \subseteq V^{\otimes n}.
\end{align*}
By \cite[Proposition 2.12]{morismith}, $\sfw_\alpha$ is a basis for the one-dimensional space $\bigcap_{i=1}^{n-1} W_i \subseteq V^{\otimes n}$, so it suffices to show that $\sfw_\alpha \in W_i$ for all $1 \leqslant i \leqslant n-1$.

Fix $i \in \{1, \dots, n-1\}$ and let $\sigma \in S_n$ be an even permutation. Note that if we define  $\sigma' \coloneqq \sigma (i\;i+1)$, then $\sgn \sigma' = -\sgn \sigma = -1$. Writing
\begin{gather*}
p \coloneqq y_1^\sigma \otimes \dots \otimes y_{i-1}^\sigma, \quad
q \coloneqq y_{i+2}^\sigma \otimes \dots \otimes y_{n}^\sigma, \quad
q' \coloneqq y_{i+2}^{\sigma'} \otimes \dots \otimes y_{n}^{\sigma'},
\end{gather*}
we have
\begin{gather*}
\sfw_\sigma = p \otimes y_i^\sigma \otimes y_{i+1}^\sigma \otimes q, \quad
\sfw_{\sigma'} = p \otimes y_{i+1}^\sigma \otimes y_{i}^\sigma \otimes q'.
\end{gather*}
We note that $q=q'$ by Lemma \ref{lem:OmegaAndPiProp} (1). Therefore
\begin{align*}
\frac{1}{\gamma_\alpha^\sigma} \sfw_\sigma - \frac{1}{\gamma_\alpha^{\sigma'}} \sfw_{\sigma'}
&= p \otimes \bigg( \frac{1}{\gamma_\alpha^\sigma} y_i^\sigma \otimes y_{i+1}^\sigma - \frac{1}{\gamma_\alpha^{\sigma'}} y_{i+1}^\sigma \otimes y_{i}^\sigma \bigg) \otimes q \\
&= \frac{1}{\gamma_\alpha^\sigma} p \otimes \bigg( y_i^\sigma \otimes y_{i+1}^\sigma - \frac{\alpha(y_i^\sigma,y_{i+1}^\sigma)}{\alpha(y_{i+1}^\sigma,y_{i}^\sigma)} y_{i+1}^\sigma \otimes y_{i}^\sigma \bigg) \otimes q
\in W_i,
\end{align*}
where the last equality uses Lemma \ref{lem:GammaProperty}. Summing over all even permutations $\sigma \in S_n$ gives
\begin{align*}
W_i \ni \sum_{\substack{\sigma \in S_n\\ \sgn \sigma = 1}} \!\!\bigg( \frac{1}{\gamma_\alpha^\sigma} \sfw_\sigma - \frac{1}{\gamma_\alpha^{\sigma'}} \sfw_{\sigma'} \bigg) = \sum_{\sigma \in S_n} \frac{\sgn \sigma}{\gamma_\alpha^\sigma} \sfw_\sigma = \sfw_\alpha,
\end{align*}
as desired.
\end{proof}

It remains to identify the twist, i.e.~the map $\theta \in \GL(V)$ which satisfies $({\id} \otimes \theta) \phi (\sfw_\alpha) = \sfw_\alpha$, where
\begin{align*}
\phi: V^{\otimes m} \to V^{\otimes m}, \qquad \phi( v_1 \otimes v_2 \otimes \dots \otimes v_{m-1} \otimes v_{m}) = v_2 \otimes v_3 \otimes \dots \otimes v_{m} \otimes v_{1}.
\end{align*}

\begin{proposition} \label{prop:Twist}
The twist $\theta \in \GL(V)$ for the twisted superpotential $\sfw_\alpha$ is
\begin{align*}
\theta(x) = (-1)^{n-1} \sgn \psi_x \frac{\alpha(x,\Delta)}{\alpha(\Delta,\psi_\delta(x))} \psi_\delta(x).
\end{align*}
\end{proposition}
\begin{proof}
Fix $\sigma \in S_n$ and consider the term $\frac{\sgn \sigma}{\gamma_\alpha^\sigma} y_1^\sigma \otimes \dots \otimes y_n^\sigma$ in $\sfw_\alpha$. We first identify a permutation $\tau \in S_n$ for which 
\begin{align*}
\phi(y_1^\sigma \otimes \dots \otimes y_n^\sigma) = y_1^\tau \otimes \dots \otimes y_{n-1}^\tau \otimes y,
\end{align*}
for some $y \in X$; the precise value of $y$, which will emerge during the proof, determines the definition of $\theta$. Write $x = x_{\sigma(1)} = y_1^\sigma$ and $z^\sigma = y_2^\sigma \dots y_n^\sigma \in M$. Using Lemma \ref{lem:PropertiesOfDelta} (3), we have
\begin{align*}
x \Delta = \Delta \psi_\delta(x) = x z^\sigma \psi_\delta(x),
\end{align*}
so $\Delta = z^\sigma \psi_\delta(x)$ since $M$ is cancellative. Write $\pi_x \in S_n$ for the permutation satisfying $\psi_x(x_i) = x_{\pi_x(i)}$. We now work towards showing that the desired permutation $\tau$ is given by
\begin{align}
\tau = \pi_x \circ \sigma \circ (1\;2\; \dots \;n). \label{eqn:DefinitionOfTau}
\end{align}
We note that this permutation has the property that
\begin{align*}
x_{\tau(i)} = \psi_x(x_{\sigma(i+1)}) \quad \text{for } 1 \leqslant i \leqslant n,
\end{align*}
where we interpret subscripts modulo $n$, where necessary.

We claim that 
\begin{align}
y_j^\tau = y_{j+1}^\sigma \quad \text{for} \quad 1 \leqslant j \leqslant n-1 \quad \text{and} \quad y_n^\tau = \psi_\delta(x), \label{eqn:yjTau}
\end{align}
which is precisely the property of $\tau$ that we want. Indeed, if $1 \leqslant j \leqslant n-1$, then
\begin{align*}
x y_2^\sigma \dots y_{j+1}^\sigma
&= y_1^\sigma y_2^\sigma \dots y_n^\sigma \\
&= \nu(x_{\sigma(1)} \dots x_{\sigma(j+1)}) \\
&= x_{\sigma(1)} \nu (\psi_x(x_{\sigma(2)} \dots x_{\sigma(j+1)})) \\
&= x_{\sigma(1)} \nu (x_{\tau(1)} \dots x_{\tau(j)}) \\
&= x y_1^\tau \dots y_{j}^\tau.
\end{align*}
Since $M$ is cancellative, we deduce that
\begin{align*}
y_2^\sigma \dots y_{j+1}^\sigma = y_1^\tau \dots y_j^\tau
\end{align*}
for all $1 \leqslant j \leqslant n-1$. Induction, using left cancellation again, shows that $y_j^\tau = y_{j+1}^\sigma$ for all $1 \leqslant j \leqslant n-1$. Using this, we also have
\begin{align*}
y_1^\tau \dots y_{n-1}^\tau y_n^\tau = \nu(x_{\tau(1)} \dots x_{\tau(n)}) = \nu(\delta) = \Delta = z^\sigma \psi_\delta(x) = y_2^\sigma \dots y_n^\sigma \psi_\delta(x) = y_1^\tau \dots y_{n-1}^\tau \psi_\delta(x),
\end{align*}
and cancellation gives $y_n^\tau = \psi_\delta(x)$, establishing \eqref{eqn:yjTau}.

Since $({\id} \otimes \theta)\phi(\sfw) = \sfw$ and
\begin{align*}
({\id} \otimes \theta)\phi\Big( \frac{\sgn \sigma}{\gamma_\alpha^\sigma} x \otimes y_2^\sigma \otimes \dots \otimes y_n^\sigma \Big) = \frac{\sgn \sigma}{\gamma_\alpha^\sigma} y_2^\sigma \otimes \dots \otimes y_n^\sigma \otimes \theta(x) = \frac{\sgn \sigma}{\gamma_\alpha^\sigma} y_1^\tau \otimes \dots \otimes y_{n-1}^\tau \otimes \theta(x),
\end{align*}
and $y_n^\tau = \psi_\delta(x)$, we must have $\theta(x) = \lambda_x \psi_\delta(x)$ for some $\lambda_x \in \Bbbk$ which we need to determine. From the above, $\lambda_x$ satisfies
\begin{align*}
\frac{\sgn \sigma}{\gamma_\alpha^\sigma} \lambda_x = \frac{\sgn \tau}{\gamma_\alpha^\tau}.
\end{align*}
Rearranging and using \eqref{eqn:DefinitionOfTau} gives
\begin{align*}
\lambda_x = \frac{\sgn \tau}{\sgn \sigma}\;\frac{\gamma_\alpha^\sigma}{\gamma_\alpha^\tau} = (-1)^{n-1} \sgn \psi_x \; \frac{\gamma_\alpha^\sigma}{\gamma_\alpha^\tau}.
\end{align*}

It remains to simplify this expression into the form given in the statement of the proposition. Write
\begin{align*}
z_j^\sigma = y_2^\sigma \dots y_j^\sigma, 1 \leqslant j \leqslant n,
\end{align*}
with the understanding that $z_1^\sigma = 1$, and so that $z_n^\sigma = z^\sigma$. By the cocycle identity, we have
\begin{align*}
\alpha(x,z_j^\sigma) \alpha(x z_j^\sigma, y_{j+1}^\sigma) = \alpha(x, z_j^\sigma y_{j+1}^\sigma) \alpha(z_j^\sigma, y_{j+1}^\sigma) = \alpha(x, z_{j+1}^\sigma) \alpha(z_j^\sigma, y_{j+1}^\sigma).
\end{align*}
Using this, we calculate
\begin{align*}
\gamma_\alpha^\sigma 
&= \prod_{j=1}^{n-1} \alpha( y_1^\sigma \dots y_j^\sigma, y_{j+1}^\sigma) 
= \prod_{j=1}^{n-1} \alpha( x z_j^\sigma, y_{j+1}^\sigma) \\
&= \prod_{j=1}^{n-1} \frac{\alpha( x, z_{j+1}^\sigma)}{\alpha(x,z_j^\sigma)} \alpha(z_j^\sigma, y_{j+1}^\sigma) 
= \frac{\alpha( x, z_{n}^\sigma)}{\alpha(x,z_1^\sigma)} \prod_{j=1}^{n-1} \alpha(z_j^\sigma, y_{j+1}^\sigma) \\
&= \alpha( x, z^\sigma)\prod_{j=2}^{n-1} \alpha(z_j^\sigma, y_{j+1}^\sigma) 
= \alpha( x, z^\sigma) \prod_{j=1}^{n-2} \alpha(z_{j+1}^\sigma, y_{j+2}^\sigma) \\
&= \alpha( x, z^\sigma) \prod_{j=1}^{n-2} \alpha(y_2^\sigma \dots y_{j+1}^\sigma, y_{j+2}^\sigma) 
= \frac{\alpha( x, z^\sigma)}{\alpha(y_2^\sigma \dots y_n^\sigma, \psi_\delta(x))} \prod_{j=1}^{n-1} \alpha(y_1^\tau \dots y_{j}^\tau, y_{j+1}^\tau) \\
&= \frac{\alpha( x, z^\sigma)}{\alpha(z^\sigma, \psi_\delta(x))} \gamma_\alpha^\tau,
\end{align*}
where we telescope to obtain the fourth equality. Finally, the cocycle identity gives
\begin{align*}
\alpha(x,z^\sigma) \alpha(\Delta, \psi_\delta(x)) &= \alpha(x,z^\sigma) \alpha(xz^\sigma,\psi_\delta(x)) = \alpha(x,z^\sigma \psi_\delta(x)) \alpha(z^\sigma,\psi_\delta(x)) \\
&= \alpha(x,\Delta) \alpha(z^\sigma,\psi_\delta(x)),
\end{align*}
so that
\begin{align*}
\frac{\gamma_\alpha^\sigma}{\gamma_\alpha^\tau} = \frac{\alpha(x,z^\sigma)}{\alpha(z^\sigma,\psi_\delta(x))} = \frac{\alpha(x,\Delta)}{\alpha(\Delta, \psi_\delta(x))}.
\end{align*}
Therefore the twist $\theta \in \GL(V)$ is given by
\begin{gather*}
\theta(x) = (-1)^{n-1} \sgn \psi_x \; \frac{\gamma_\alpha^\sigma}{\gamma_\alpha^\tau} = (-1)^{n-1} \sgn \psi_x \; \frac{\alpha(x,\Delta)}{\alpha(\Delta, \psi_\delta(x))} \psi_\delta(x).\qedhere
\end{gather*}
\end{proof}

In \cite[Theorem 4.4]{morismith}, the authors prove that the \emph{Nakayama automorphism} $\mu_A$ of a Koszul AS regular algebra $A$ of dimension $n$ is given by $\mu_A = (-1)^{n-1} \theta_A$, where $\theta_A$ is the twist for the defining superpotential for $A$. (We caution the reader that we use a different convention for the map $\phi : V^{\otimes n} \to V^{\otimes n}$ compared with \cite{morismith}, and we take our definition for the Nakayama automorphism as in \cite{rrz}. With these conventions, the Nakayama automorphism does satisfy the claimed identity given above.) In particular, $A$ is Calabi--Yau if and only if $\theta = (-1)^{n-1} \id_V$.

\begin{proposition} \label{prop:CY}
The twisted structure algebra $\Bbbk^\alpha M$ is Calabi--Yau if and only if 
\begin{align*}
\Delta \in Z(M) \quad \text{and} \quad \sgn \psi_x = \frac{\alpha(x,\Delta)}{\alpha(\Delta,x)} \quad \text{for all } x \in X.
\end{align*}
If $\alpha$ is trivial, this happens if and only if $\Delta \in Z(M)$ and $\psi_x$ is even for each $x \in X$.
\end{proposition}
\begin{proof}
By combining the preceding discussion with Proposition \ref{prop:Twist}, we find that $\Bbbk^\alpha M$ is Calabi--Yau if and only if 
\begin{align*}
\sgn \psi_x \frac{\alpha(x,\Delta)}{\alpha(\Delta,\psi_\delta(x))} \psi_\delta(x) = x
\end{align*}
for all $x \in X$, which happens if and only if
\begin{align*}
\sgn \psi_x \frac{\alpha(x,\Delta)}{\alpha(\Delta,\psi_\delta(x))} = 1 \quad \text{and} \quad \psi_\delta(x) = x \quad \text{for all } x \in X.
\end{align*}
Since $x \Delta = \Delta \psi_\delta(x)$, this is equivalent to the condition given in the statement.
\end{proof}

\begin{example} \label{ex:BuzagloCY}
We return to Example \ref{ex:BuzagloExample}, in which $X = \{x,y,z\}$ and $L_a = (x\;z\;y) \eqqcolon \sigma$ for all $a \in X$. We seek to determine the defining superpotential for $\Bbbk^\alpha M$ and the corresponding twist, where $\alpha$ is the $2$-cocycle defined earlier.

In this example, the $6=3!$ factorisations of $\Delta$ are easily determined, since $\Pi_3(a,b,c) = (a, \sigma(a), \sigma^2(a))$. We therefore have
\begin{gather*}
\begin{array}{ccc}
\Pi_3(x,y,z) = (x,x,x), & \Pi_3(x,z,y) = (x,y,z), & \Pi_3(y,x,z) = (y,z,x), \\
\Pi_3(y,z,x) = (y,y,y), & \Pi_3(z,x,y) = (z,z,z), & \Pi_3(z,y,x) = (z,x,y).
\end{array}
\end{gather*}
Taking into account the sign of the permutation corresponding to the order of $x,y,$ and $z$ in the argument of $\Pi_3$, Proposition \ref{prop:Superpotential} tells us that the twisted superpotential for $\Bbbk M$ is
\begin{align*}
\sfw = x^3+y^3+z^3 - xyz - yzx - zxy \in V^{\otimes 3},
\end{align*}
where $V = \spann X$ and we suppress the tensor symbols. 

If we wish to determine the superpotential when $\alpha$ is nontrivial, then we need to calculate the terms $\gamma_\alpha(a,b,c)$, where $abc$ is a factorisation of $\Delta$. We have
\begin{align*}
\gamma_\alpha(x,x,x) &= \alpha(x,x) \alpha(x^2,x) = (-q)^{B(x,y)} (-q)^{B(xy,z)} = (-q)^2 = q^2, \\
\gamma_\alpha(x,y,z) &= \alpha(x,y) \alpha(xy,z) = (-q)^{B(x,z)} (-q)^{B(xz,y)} = 1.
\end{align*}
We obtain the coefficients for the remaining terms by essentially the same calculations. Therefore, using Proposition \ref{prop:Superpotential} and harmlessly multiplying through by $q^2$ gives a superpotential
\begin{align*}
\sfw_\alpha = x^3+y^3+z^3 + qxyz + qyzx + qzxy \in V^{\otimes 3}.
\end{align*}

We now determine the twist $\theta$ which, by Proposition \ref{prop:Twist} is given by
\begin{align*}
\theta(a) = (-1)^{n-1} \sgn \psi_a \frac{\alpha(a,\Delta)}{\alpha(\Delta,\psi_\delta(a))} \psi_\delta(a).
\end{align*}
As we have already seen, the Garside element $\Delta \in M$ satisfies
\begin{align*}
\Delta = x^3 = y^3 = z^3,
\end{align*}
and hence is central in $M$. Therefore, by Lemma \ref{lem:PropertiesOfDelta} (3), $\psi_\delta = \id_X$. We also have $\psi_a = L_a = (x\;z\;y)$ for all $a \in X$, which is an even permutation. Finally, for all $a \in X$ we have
\begin{gather*}
\alpha(a,\Delta) = (-q)^{B(a,(x\;y\;z) \cdot xyz)} = -q, \\
\alpha(\Delta,a) = (-q)^{B(xyz,(x\;y\;z)^3 \cdot a)} = -q.
\end{gather*}
It follows that $\theta = \id_V$, and hence $\Bbbk^\alpha M$ is Calabi--Yau by \cite[Theorem 4.4]{morismith}, or by applying Proposition \ref{prop:CY}.
\end{example}

\subsection{The homological determinant}
The \emph{homological determinant}, $\operatorname{hdet}_A : H \to \Bbbk$, is an important invariant associated to the action of a Hopf algebra $H$ on an AS Gorenstein algebra $A$. It is a generalisation of the matrix determinant in the sense that, if $g \in \GL(V)$ acts naturally on the polynomial ring $R = \Sym_\Bbbk(V)$, then $\operatorname{hdet}_R(g) = \det(g)$. The homological determinant controls a number of important properties about $A$ under the action from $H$, such as when the invariant ring $A^H$ is AS Gorenstein, or when the so-called \emph{Auslander map} is an isomorphism. In general, the homological determinant is difficult to calculate but, for a superpotential algebra $A = \mathscr{D}(\sfw,i)$ it is known, by \cite{morismith,crawfordsuperpotentials}, that we have
\begin{align*}
h \cdot \sfw = \operatorname{hdet}_A(h) \sfw \quad \text{for all } h \in H.
\end{align*}

\begin{lemma}
Let $A = \Bbbk^\alpha M$ and $G_k = G(X,r,k)$ be a finite Dehornoy quotient group, so that $A$ is $G_k$-graded. The corresponding action of $H = (\Bbbk G_k)^*$ has homological determinant $\operatorname{hdet}_A : A \to \Bbbk$ satisfying 
\begin{align*}
\operatorname{hdet_A}(\phi_g) = \left\{
\begin{array}{cl}
1 & \text{if } g = \pi_k(\Delta), \\
0 & \text{otherwise},
\end{array}
\right.
\end{align*}
where $\{\phi_g \mid g \in G\}$ is the basis of $H$ dual to the standard basis of $\Bbbk G$.
\end{lemma}
\begin{proof}
This follows immediately from \cite[Corollary 3.3]{crawfordsuperpotentials}, since $\sfw_\alpha$ is a linear combination of all factorisations of $\Delta$ in $M$.
\end{proof}


\section{Examples and non-examples}

\subsection{Examples in the recent literature}
As mentioned in the introduction, the results of \cite{gateva,dehornoy} and the present article are applicable to a number of recent papers which study of dual reflection groups. We already saw in Example \ref{ex:BuzagloExample} that \cite[Theorem 6.3]{buzaglo} follows quickly from our results. We now draw attention to some further recent examples in the literature.

\begin{example} \label{ex:Goetz1}
In \cite[Section 3.1]{goetz}, the authors consider, albeit with a different presentation, the group
\begin{align*}
G = M_{16} = \langle \, g,h \mid hg=g^5h, \; g^8=h^2=1 \,\rangle,
\end{align*}
and the algebra
\begin{align*}
R = \frac{\Bbbk\langle x_1,x_2,x_3,x_4\rangle}{
\left\langle
\begin{array}{ccc}
x_1^2=x_4^2 & x_1 x_4 = x_3 x_2 & x_1 x_3 = x_2 x_4 \\
x_2^2=x_3^2 & x_4 x_1 = x_2 x_3 & x_3 x_1 = x_4 x_2 
\end{array}
\right\rangle
}
\end{align*}
which is $G$-graded via
\begin{align*}
\deg_G x_1 = g, \quad \deg_G x_2 = g^3, \quad \deg_G x_3 = gh, \quad \deg_G x_4 = g^7h. 
\end{align*}
They go on to show that $R$ is AS regular and that $R_{1_G}$ is a polynomial ring in four variables, so that $G$ is a dual reflection group for $R$.

One can instead prove this result by appealing to Theorems \ref{thm:Gateva} and \ref{thm:DualReflection}. The solution to the Yang--Baxter equation, or equivalently, the cycle set, which gives rise to this algebra can be deduced from the relations in $A$. Indeed, the set $X = \{x_1, x_2, x_3, x_4\}$ is a cycle set with
\begin{align*}
L_{x_1} = L_{x_3} = (x_1\;x_2\;x_3\;x_4), \quad 
L_{x_2} = L_{x_4} = (x_1\;x_4\;x_3\;x_2),
\end{align*}
such that, writing $M$ for the structure monoid associated to this solution, we have $R = \Bbbk M$. This algebra is AS regular of dimension $4$ by Theorem \ref{thm:Gateva}. Moreover, this solution is of minimal Dehornoy class $2$, and it is straightforward to calculate that the finite Dehornoy quotient groups are 
\begin{align*}
G_n = \left\langle a,b,c,d \hspace{4pt} \bigg| \hspace{-1pt}
\begin{array}{cccc}
a^2=d^2 & ad=cb & ac=bd & (ab)^n=(ba)^n=1 \\
b^2=c^2 & da=bc & ca=db & (cd)^n=(dc)^n=1
\end{array} \hspace{-1pt}
\right\rangle
\end{align*}
of order $16n^2$. By Theorem \ref{thm:DualReflection}, the subring of dual invariants is a polynomial ring in four variables. When $n=1$, this group is isomorphic to $M_{16}$, so we recover \cite[Proposition 3.1.2]{goetz}. For $n \geqslant 2$, this result was proved in \cite[Proposition 7.5]{buzaglo}.
\end{example}

\begin{example} \label{ex:Goetz2}
In \cite[Section 3.2]{goetz}, the authors consider, again with an alternative presentation, the group
\begin{align*}
G = SD_{16} = \langle\, g,h \mid hg=g^3h,\; g^8=h^2=1 \,\rangle,
\end{align*}
as well as the algebra
\begin{align*}
S = \frac{\Bbbk\langle x_1,x_2,x_3,x_4\rangle}{
\left\langle
\begin{array}{ccc}
x_1x_2=x_3^2 & x_1 x_3 = x_2 x_4 & x_1 x_4 = x_4 x_2 \\
x_2x_1=x_4^2 & x_3 x_2 = x_4 x_1 & x_2 x_3 = x_3 x_1 
\end{array}
\right\rangle
}
\end{align*}
which has a $G$-grading given by
\begin{align*}
\deg_G x_1 = h, \quad \deg_G x_2 = g^6h, \quad \deg_G x_3 = g, \quad \deg_G x_4 = g^7. 
\end{align*}
The algebra $S$ was mentioned in \cite[Remark 4.2 (d)]{tomlin} (see \cite[p.~2]{goetz}), and whether or not it was AS regular was unknown for a period of time. A proof of its AS regularity is given in \cite[Theorem 3.2.8]{goetz}. Presumably unknown to these authors, a monoid with the same relations was also considered in \cite[Example 1.3]{chouraqui}. In particular, there is a solution to the Yang--Baxter equation for which the associated structure algebra is $S$, and hence $S$ is AS regular by \cite[Theorem 1.5]{gateva}.

In this case, the finite Dehornoy quotient groups are
\begin{align*}
G_n = \left\langle a,b,c,d \hspace{4pt} \bigg| \hspace{-1pt}
\begin{array}{cccc}
ab=c^2 & ac=bd & ad=db & (a^2)^n=(b^2)^n=1 \\
ba=d^2 & cb=da & bc=ca & (cd)^n=(dc)^n=1
\end{array} \hspace{-1pt}
\right\rangle
\end{align*}
of order $16n^2$. Again, the subring of dual invariants is a polynomial ring in four variables by Theorem \ref{thm:DualReflection}, and the $n=1$ case recovers \cite[Proposition 3.2.10]{goetz}. As in the previous example, the $n \geqslant 2$ case is \cite[Proposition 7.5]{buzaglo}.

The authors of \cite{goetz} also consider a deformation $T$ of $S$, with presentation
\begin{align*}
T = \frac{\Bbbk\langle x_1,x_2,x_3,x_4\rangle}{
\left\langle
\begin{array}{ccc}
x_1x_2=x_3^2 & x_1 x_3 = x_2 x_4 & x_1 x_4 = -x_4 x_2 \\
x_2x_1=-x_4^2 & x_3 x_2 = x_4 x_1 & x_2 x_3 = x_3 x_1 
\end{array}
\right\rangle
}.
\end{align*}
Writing $M$ for the structure monoid associated to the corresponding solution to the Yang--Baxter equation, we claim that $T \cong \Bbbk^\alpha M$ for a suitably chosen $2$-cocycle $\alpha : M \times M \to \Bbbk^\times$. To construct this cocycle, we note that the relations in $S$ which are deformed to give $T$ are 
\begin{align*}
\Pi_2(x_1,x_4) = \Pi_2(x_4,x_1) \quad \text{and} \quad \Pi_2(x_2,x_4) = \Pi_2(x_4,x_2).
\end{align*}
This suggests that $\alpha$ should identify when an element of $X^m$ has copies of $x_1$ or $x_2$ appearing to the right of copies of $x_4$. To this end, we modify the inversion number used in the definition of $\rho$ as follows. If $C = (x_{k_1}, \dots, x_{k_m}) \in X^m$ then we define
\begin{align*}
\operatorname{inv}_T(C) \coloneqq |\{(i,j) \mid i>j \text{ and } x_{k_i} \in \{x_1, x_2\} \text{ and } x_{k_j}=x_4 \}|.
\end{align*}
We claim the the following defines a $2$-cocycle on $M$:
\begin{align*}
\alpha : M \times M \to \Bbbk^\times, \quad \alpha(a,b) = (-1)^{\operatorname{inv}_T\!\big(\Pi^{-1}_{\ell(ab)}(\NF(a) \frown \NF(b))\big)}.
\end{align*}
One can show that this is a $2$-cocycle by using an argument similar to the one in Example \ref{ex:NormalForm}. Moreover, we have, for example
\begin{gather*}
\alpha(x_2,x_1) = (-1)^{\operatorname{inv}_T\!\big(\Pi^{-1}_{2}(x_2,x_1)\big)} = (-1)^{\operatorname{inv}_T(x_2,x_4)} = 1 \\
\alpha(x_4,x_4) = (-1)^{\operatorname{inv}_T\!\big(\Pi^{-1}_{2}(x_4,x_4)\big)} = (-1)^{\operatorname{inv}_T(x_4,x_2)} = -1,
\end{gather*}
so that $x_2x_1 = -x_4x_4$ in $\Bbbk^\alpha M$. One can check that the other relations of $T$ hold, so that $\Bbbk^\alpha M \cong T$. It follows that $T$ is AS regular, and that $G_n$ is a dual reflection group for $T$, recovering \cite[Theorem 3.2.8, Proposition 3.2.11]{goetz} and \cite[Proposition 7.5]{buzaglo}. In particular, since the subring of dual invariants is not commutative when $n=1$, we know that $\alpha$ does not descend to a $2$-cocycle on $G_1 = SD_{16}$.
\end{example}

\subsection{Dihedral Dehornoy quotient groups}
In \cite[Example 2.1.9]{goetz}, the authors show that the dihedral group of order $8$, which we denote $D_4$, is a dual reflection group. They then remark that they are not aware of any examples of dihedral groups $D_m$ of order $>8$ which are dual reflection groups. We show that there are no such examples among twisted structure algebras.

We first require an easy lemma.

\begin{lemma} \label{lem:SnUpperBound}
Suppose that $\tau \in S_n$. Then $\tau$ has order at most $3^{n/3}$.
\end{lemma}
\begin{proof}
Suppose that $\tau$ is a product of disjoint cycles of lengths $\ell_1, \dots, \ell_m$, with $\ell_1 + \dots +\ell_m = n$. Since $\ell \leqslant 3^{\ell/3}$ for all $\ell \in \NN$, we have
\begin{align*}
|\tau| =& \operatorname{lcm}(\ell_1, \dots, \ell_m) \leqslant \ell_1 \dots \ell_m \leqslant 3^{\ell_1/3} \dots 3^{\ell_m/3} = 3^{(\ell_1+\dots+\ell_m)/3} = 3^{n/3}. \qedhere
\end{align*}
\end{proof}

The above upper bound is crude, but sufficient for our purposes.

\begin{proposition}
Suppose that $r : X^2 \to X^2$ is a solution to the Yang--Baxter equation of Dehornoy class $c$ and that $G(X,r,k)$ is dihedral for some $k \geqslant 1$. Then $|X| = 3$, $c=2$, and $k=1$.
\end{proposition}
\begin{proof}
Write $n=|X|$, $d=ck$ and $G_k = G(X,r,k)$. If we had $c=1$, then $r$ would be the trivial flip solution and hence $G$ would be abelian, so we must have $c \geqslant 2$, and hence $d \geqslant 2$.

Suppose that 
\begin{align*}
G_k \cong D_m \coloneqq \langle\,g,h \mid hg=g^{m-1}h,\; g^m=h^2=1 \;\rangle
\end{align*}
for some $m \geqslant 2$, where $|D_m| = 2m$. Therefore $2m = d^n$, so $|g| = m = \frac{1}{2} d^n$.

By \cite[Corollary 5.11]{dehornoy}, there is an embedding $G_k \hookrightarrow (\ZZ/d\ZZ)^n \rtimes S_n$ of $G_k$ into a wreath product. Suppose that the element of $G_k$ corresponding to $g \in D_m$ has image $(v,\tau)$ under this embedding. Then
\begin{align*}
(v,\tau)^{d |\tau|} = (v',1)^d = (1,1),
\end{align*}
so $|g| \leqslant d |\tau| \leqslant d \, 3^{n/3}$, by Lemma \ref{lem:SnUpperBound}. Combining this with the previous paragraph gives 
\begin{align*}
\frac{1}{2} d^n \leqslant d\, 3^{n/3},
\end{align*}
and hence
\begin{align*}
2^{n-2} \leqslant \frac{1}{2} d^{n-1} \leqslant 3^{n/3}.
\end{align*}
This forces $n \leqslant 4$. 

If $n=4$, then $\tau \in S_4$ so $|\tau| \leqslant 4$, which means
\begin{align*}
\frac{1}{2} d^4 = |g| \leqslant d |\tau| \leqslant 4d,
\end{align*}
and hence $d \leqslant 2$. Since $c\geqslant 2$, we have $c=2$ and $k=1$. In particular, the embedding $G_1 \hookrightarrow (\ZZ/2\ZZ)^4 \rtimes S_4$ gives $(\ZZ/2\ZZ)^4$ the structure of a so-called \emph{right brace} whose multiplicative group is $D_8$, but this is impossible by \cite[Proposition 4]{rump2020}.

If $n=2$, we instead get $d \leqslant 4$, so that $d=2$ or $d=4$. If $d=4$, then the embedding $G_k \hookrightarrow (\ZZ/4\ZZ)^2 \rtimes S_2$ gives $(\ZZ/4\ZZ)^2$ the structure of a right brace with multiplicative group $D_8$, which is again impossible by \cite[Proposition 2]{rump2020}. Finally, if $n \leq 2$ then $d=2$, so $|G_k| = 4$ and $G_k$ is not dihedral.

Finally, if $n=3$, then the same argument as above gives $d \leqslant 2$, so $c=2$ and $k=1$. By \cite[Example 2.1.9]{goetz}, this scenario can occur; for completeness, we list the possibilities. There are exactly five cycle set structures on $X=\{x,y,z\}$ (equivalently, solutions to the Yang--Baxter equation on a set of size 3), up to isomorphism. Of these, when $c=2$ and $k=1$, only two give $G_1 \cong D_4$. They are
\begin{align*}
\begin{array}{ccc|c|c}
L_x & L_y & L_z & \text{Relations in } M(X,r) & \text{Additional relations in } G_1 \\\hline
& & & & \\[-10pt]
\id & \id & (x\;y) & xy=yx,\;\;xz=zy,\;\;yz=zx & x^2=y^2=z^2=1\\
(y\;z) & (y\;z) & (y\;z) & xz=yx,\;\;xy=zx,\;\;y^2=z^2 & x^2=yz=zy=1
\end{array}
\end{align*}
The first of these corresponds to the group in \cite[Example 2.1.9]{goetz}, while the second gives a different $D_4$ coaction.
\end{proof}


\bibliographystyle{amsalpha}
\bibliography{bibliography}
\end{document}